\documentclass[11pt,leqno]{article}
\usepackage{amsmath,amssymb,amsthm,latexsym}
\usepackage{indentfirst}
\usepackage{indentfirst}
\usepackage[all]{xy}
\usepackage{hyperref}
\usepackage{fancyhdr}
\usepackage{indentfirst}
\usepackage{titlesec}
\usepackage{listings}
\usepackage{cases}
\usepackage{graphics}
\usepackage{graphicx}
\usepackage{times}

\usepackage{amsmath}
\newtheorem{Theorem}{\bf \large Theorem}[section]
\newtheorem{Definition}[Theorem]{\bf \large Definition}
\newtheorem{PROPOSITION}[Theorem]{\bf \large Proposition}

\newtheorem{lemma}[Theorem]{\bf \large Lemma}
\newtheorem{ex}[Theorem]{\bf \large Example}
\newtheorem{Remark}[Theorem]{\bf \large Remark}

\def\e{\rm e}
\title{\textbf{Classification of M\"{o}bius Homogeneous Wintgen Ideal Submanifolds}}
\author {Tongzhu Li \and Xiang Ma
\and Changping Wang \and Zhenxiao Xie}
\date{}

\begin{document}
\maketitle  \footnotetext{T. Li, X. Ma, C.P. Wang are partially supported
by the grant No. 11171004 of NSFC.}
\begin{abstract}
A submanifold $f:M^m\to \mathbb{Q}^{m+p}(c)$
in a real space form attaining equality in the DDVV inequality at every point is called a Wintgen ideal submanifold.
They are invariant objects under the M\"{o}bius transformations. In this paper, we classify those Wintgen ideal submanifolds of dimension $m\geq3$ which are M\"{o}bius homogeneous.
There are three classes of non-trivial examples, each related with a famous class of homogeneous minimal surfaces in $S^n$ or $\mathbb{C}P^n$: the cones over Veronese surfaces $S^2\to S^n$, the cones over homogeneous flat minimal surfaces $\mathbb{R}^2\to S^n$, and the Hopf bundle over the Veronese embeddings $\mathbb{C}P^1\to \mathbb{C}P^n$.
\end{abstract}

\vskip 1 cm
\section{Introduction}

A central theme in geometry is to find and characterize those best shapes. It often means to find the optimally immersed submanifolds in a fixed ambient space. Two widely used optimality criteria are the minimization of certain functional(s), and the existence of many symmetries.

From this viewpoint, homogeneous minimal surfaces in real space forms are best submanifolds, which include the Veronese surfaces $S^2\to S^{2k}$ and Clifford type surfaces $\mathbb{R}^2\to S^{2k+1}$ \cite{br1, li1}. In complex space forms there are similar examples \cite{BJ, Calabi}.

In this paper we will consider M\"obius homogeneous, Wintgen ideal submanifolds in M\"obius geometry, which might be regarded also as best submanifolds according to both criteria. To our happy surprise, the classification shows that they are closely related with those homogeneous minimal surfaces mentioned above.

To explain what is a \emph{Wintgen ideal submanifold}, note that an optimality criterion is to consider a universal inequality and find out all the cases when the equality is achieved, which is somewhat similar to the minimization criterion. In this spirit, we are interested in the equality case in the so-called \emph{DDVV inequality} \cite{DDVV1} for a generic submanifold $f:M^m\to \mathbb{Q}^{m+p}(c)$ in a real space form. This inequality is remarkable because it relates the most important intrinsic and extrinsic quantities at an arbitrary point $x\in M$, without any restriction on the dimension/codimension or any further geometric/topological assumptions. This universal inequality was a difficult conjecture in \cite{DDVV1, DDVV}, and was finally proved in \cite{Ge} and \cite{Lu3}. By the suggestion of \cite{chen10,ml} and the characterization of \cite{Ge} about the equality case at an arbitrary point, we make the following definition.

\begin{Definition}
We denote $f:M^m\longrightarrow \mathbb{Q}^{m+p}(c)$ a submanifold of dimension $m$ and codimension $p$ in a real space form of constant sectional curvature $c$. It is a \emph{Wintgen ideal submanifold} if the equality is attained at every point of $M^m$ in the DDVV inequality. This happens if, and only if, at every point $x\in M$ there exists an orthonormal basis $\{e_1,\cdots,e_m\}$ of the tangent plane $T_xM^m$ and an orthonormal basis $\{n_1,\cdots,n_p\}$ of the normal plane $T_x^{\bot}M^m$,
such that the shape operators $\{A_{n_i},i=1,\cdots,p\}$ take the form as below \cite{Ge}:
\begin{equation}\label{form1}
A_{n_1}=
\begin{pmatrix}
\lambda_1 & \mu_0 & 0 & \cdots & 0\\
\mu_0 & \lambda_1 & 0 & \cdots & 0\\
0  & 0 & \lambda_1 & \cdots & 0\\
\vdots & \vdots & \vdots & \ddots & \vdots\\
0  & 0 & 0 & \cdots & \lambda_1
\end{pmatrix},
A_{n_2}=
\begin{pmatrix}
\lambda_2+\mu_0 & 0 & 0 & \cdots & 0\\
0 & \lambda_2-\mu_0 & 0 & \cdots & 0\\
0  & 0 & \lambda_2 & \cdots & 0\\
\vdots & \vdots & \vdots & \ddots & \vdots\\
0  & 0 & 0 & \cdots & \lambda_2
\end{pmatrix},
\end{equation}
$$A_{n_3}=\lambda_3I_m,~~~~ A_{n_r}=0, r\ge 4.$$
\end{Definition}

Wintgen ideal submanifolds are abundant. Wintgen first proved the DDVV inequality for surfaces in $\mathbb{S}^4$, and characterized the equality case \cite{wint}. More general, a surface $f:M^2\longrightarrow \mathbb{Q}^{2+p}_c$ of arbitrary codimension $p$ is Wintgen ideal exactly when the curvature ellipse is a circle at every point \cite{gr}, which is also equivalent to the Hopf differential being isotropic. For more examples see \cite{br,Dajczer,Dajczer1,LiTZ2,Lu}.

An important observation \cite{DDVV, Dajczer1} is that the DDVV inequality as well as the equality case are invariant under M\"{o}bius transformations of the ambient space. Thus it is appropriate to put the study of Wintgen ideal submanifolds in the framework of M\"{o}bius geometry. It follows that Wintgen ideal submanifolds in the sphere $\mathbb{S}^{m+p}$ or hyperbolic space $\mathbb{H}^{m+p}$ are the pre-image of a stereographic projection of Wintgen ideal submanifolds in $\mathbb{R}^{m+p}$. For the same reason it is no restriction when we describe them in the Euclidean space.

Since there are still many possible examples of Wintgen ideal submanifolds up to M\"obius transformatons, it is natural to restrict to the best examples with many symmetries, called \emph{M\"{o}bius homogeneous submanifolds}. This means that for $f:M^m\longrightarrow \mathbb{Q}^{m+p}(c)$ and arbitrary two points $x_1,x_2\in M^m$, there exists a M\"{o}bius transformation $\phi$ of $\mathbb{Q}^{m+p}(c)$ satisfying $\phi\circ f(x_1)=f(x_2)$
and $\phi\circ f(M^m)=f(M^m)$. Such a submanifold is an orbit of a subgroup in the M\"{o}bius transformation group.

Our goal in this paper is to classify M\"{o}bius homogeneous Wintgen ideal submanifolds of dimension $m\geq 3$ in $\mathbb{R}^{m+p}$. Below are some examples.

\begin{ex}\label{ex1}
Let $f: S^2\to S^{2m} (m\geq 2)$ be one of the Veronese surfaces mentioned at the beginning. As is well-known, such examples are totally isotropic and homogeneous with respect to the isometry group of $S^{2m}$. So they are M\"{o}bius homogeneous Wintgen ideal submanifolds. They come from the irreducible orthogonal representations of $SO(3)$.
\end{ex}

\begin{ex}\label{ex2}
Let $f:\mathbb{R}^2\to S^{2m-1}$ be a Clifford-type surface. That means it is homogeneous, flat, and minimal. It comes from a subgroup of the maximal torus group $T^m\to SO(2m)$. Following \cite{br1}, it is given by $f=(f^1,\cdots,f^{2m})$
\begin{equation}\label{eq-clifford}
\begin{split}
&f^{2k-1}(x,y)=r_k\cos(x\cos\theta_k+y\sin\theta_k),\\
&f^{2k}(x,y)=r_k\sin(x\cos\theta_k+y\sin\theta_k),~~~(1\leq k\leq m)
\end{split}
\end{equation}
where $(r_k,\theta_k)$ are $m$ real numbers satisfying $r_k>0$, $\{\theta_k\}$ are distinctive modulo $k\pi$, and \[r_1^2+\cdots+r_m^2=1, ~~ \e^{2i\theta_1}r_1^2+\cdots+\e^{2i\theta_m}r_m^2=0.\]
Clearly, the flat minimal surface $f:R^2\to S^{2m-1}$ is an orbit of a $2$-dimensional abelian subgroup of $SO(2m)$. By direct computation we know that
$f:R^2\to S^{2m-1}$ is Wintgen ideal if, and only if,
\[\e^{4i\theta_1}r_1^2+\cdots+\e^{4i\theta_m}r_m^2=0.\]
\end{ex}

\begin{ex}\label{ex3}
Let $f: \mathbb{C}P^1\to \mathbb{C}P^m$ be a Veronese 2-sphere \cite{BJ}.
Let $\pi: S^{2m+1}\to \mathbb{C}P^m$ be the projection map of the Hopf bundle. Then
$\pi^{-1}\circ f:\mathbb{C}P^1\to S^{2m+1}$ is a M\"{o}bius homogeneous Wintgen ideal submanifold \cite{DDVV1}.

It comes from the irreducible unitary representations of $SU(2)$. When $m$ is an even number, this submanifold factors as an embedded $SO(3)=\mathbb{R}P^3$;
otherwise it is an embedded $SU(2)=S^3$.
\end{ex}

\begin{ex}\label{ex4}
The cone over an immersed submanifold $u:M^r\longrightarrow S^{r+p}\subset \mathbb{R}^{r+p+1}$ is
\begin{equation*}
\begin{split}
&f:R^+\times \mathbb{R}^{m-r-1}\times M^r\longrightarrow \mathbb{R}^{m+p},\\
&~~~~~~f(t,y,u)=(y,tu),
\end{split}
\end{equation*}
It is a Wintgen ideal submanifold if (and only if) $u$ is a minimal Wintgen ideal submanifold in $\mathbb{S}^{r+p}$. (See Section~4 for the proof.)
\end{ex}

Our main theorem is as below.
\begin{Theorem}\label{the1}
Let $f:M^m\longrightarrow \mathbb{R}^{m+p} (m\geq3)$ be a M\"{o}bius homogeneous Wintgen ideal submanifold. Then locally $f$ is M\"{o}bius equivalent to
\begin{description}
\item (i) a cone over a Veronese surface in $S^{2k}$,
\item (ii) a cone over a Clifford-type surface in $S^{2k+1}$,
\item (iii) or a cone over $\pi^{-1}\circ f: \mathbb{C}P^1\to S^{2k-1}$,
\item (iv) or an affine subspace in $\mathbb{R}^{m+p}$.
\end{description}
\end{Theorem}

Our conclusions do not extend to M\"obius homogeneous Wintgen ideal surfaces, i.e., when $m=2$. In $S^4$, any of them is M\"obius equivalent to part of the Veronese surface. This follows from the classification of Willmore surfaces with constant M\"obius curvature \cite{MaWang}, or from an unpublished old manuscript by H. Li, F. Wu and the third author which gave a classification of all M\"obius homogeneous surfaces in $S^4$). On the other hand, we can modify Example~\ref{ex2} to obtain homogeneous, Wintgen ideal (i.e., isotropic), isometric immersions $\mathbb{R}^2$ in $S^5$ which are not minimal. Whether there exist other kinds of examples are still unknown to us.\\

In the rest part of this introduction, we give an overview of the proof and the whole structure of this paper.

We start by reviewing the submanifold theory in M\"{o}bius geometry in Section~2. In Section~3 we restrict to Wintgen ideal submanifolds. Due to the specific, simple structure of the (M\"obius) second fundamental form, we derive the explicit expressions of the M\"{o}bius invariants.

From the statement of the main theorem one can see the importance of the construction by cones (Example~\ref{ex4}). This is described in detail in Section~4. In particular, we show that the cone $f$ is Wintgen ideal if and only if the original submanifold $u$ in the sphere is minimal and Wintgen ideal.

In Section~5, we start to utilize the assumption of M\"obius homogeneity. The first structure result is that when the dimension $m\geq 3$, the M\"{o}bius form $\Phi$ of a M\"{o}bius homogeneous Wintgen ideal submanifold vanishes. This is proved by contradiction and detailed analysis of the M\"obius invariants using the integrable equations.

One crucial ingredient in the discussions of Section~5 and 6 is to re-choose the tangent and normal frames so that the normal connection takes an elegant form. This is the main content of Lemma~\ref{le4}, Proposition~\ref{integ0} and \ref{3d}.

In Section~7 we prove a somewhat surprising reduction result, namely, if the dimension $m\geq 3$, the M\"{o}bius homogeneous Wintgen ideal submanifold is a cone over a surface or a three dimensional submanifold in $\mathbb{S}^{m+p}$.

In Section~8, we give the proof of our classification theorem. In particular, the only three dimensional M\"obius homogeneous and Wintgen ideal examples which are not cones over surfaces must be given by the Hopf bundle over the Veronese surfaces in $\mathbb{C}P^n$.

 \vskip 1 cm
\section{Submanifolds theory in M\"obius geometry}

In this section we briefly review the theory of submanifolds
in M\"obius geometry. For details we refer to \cite{CPWang} and \cite{liu}.

Let $\mathbb{R}^{m+p+2}_1$ be the Lorentz space with inner
product $\langle \cdot,\cdot\rangle $ defined by
\[
\langle Y,Z\rangle =-Y_0Z_0+Y_1Z_1+\cdots+Y_{m+p+1}Z_{m+p+1},
\]
where $Y=(Y_0,Y_1,\cdots,Y_{m+p+1}),Z=(Z_0,Z_1,\cdots,Z_{m+p+1})\in
\mathbb{R}^{m+p+2}$.

Let $f:M^m\rightarrow \mathbb{R}^{m+p}$ be a submanifold without umbilics and
assume that $\{e_i\}$ is an orthonormal basis with respect to the
induced metric $I=df\cdot df$ with $\{\theta_i\}$ the dual basis.
Let $\{n_{\alpha}|1\le \alpha\le p\}$ be a local orthonormal basis for the
normal bundle. As usual we denote the second fundamental form and
the mean curvature of $f$ as
\[
II=\sum_{ij,\alpha}h^{\alpha}_{ij}\theta_i\otimes\theta_j n_{\alpha},
~~H=\frac{1}{m}\sum_{j,\alpha}h^{\alpha}_{jj}n_{\alpha}
=\sum_{\alpha}H^{\alpha}n_{\alpha}.
\]
We define the M\"{o}bius position vector $Y:
M^m\rightarrow \mathbb{R}^{m+p+2}_1$ of $f$ by
\[
Y=\rho\left(\frac{1+|f|^2}{2},\frac{1-|f|^2}{2},f\right),~~
~~\rho^2=\frac{m}{m-1}\left|II-\frac{1}{m} tr(II)I\right|^2~.
\]
It is known that $Y$ is a well-defined canonical lift of $f$. Two submanifolds $f,\bar{f}: M^m\rightarrow \mathbb{R}^{m+p}$
are M\"{o}bius equivalent if there exists $T$ in the Lorentz group
$\mathbb{O}(m+p+1,1)$ in $\mathbb{R}^{m+p+2}_1$ such that $\bar{Y}=YT.$ It follows
immediately that
\[
g=\langle dY,dY\rangle =\rho^2 df\cdot df
\]
is a M\"{o}bius invariant, called the M\"{o}bius metric of $f$.

Let $\Delta$ be the Laplacian with respect to $g$. Define
\[
N=-\frac{1}{m}\Delta Y-\frac{1}{2m^2}\langle \Delta Y,\Delta Y\rangle Y,
\]
which satisfies
\[
\langle Y,Y\rangle =0=\langle N,N\rangle , ~~\langle N,Y\rangle =1~.
\]
Let $\{E_1,\cdots,E_m\}$ be a local orthonormal basis for $(M^m,g)$
with dual basis $\{\omega_1,\cdots,\omega_m\}$. Write
$Y_j=E_j(Y)$. Then we have
\[
\langle Y_j,Y\rangle =\langle Y_j,N\rangle =0, ~\langle Y_j,Y_k\rangle =\delta_{jk}, ~~1\leq j,k\leq m.
\]
We define
\[
\xi_\alpha=H^\alpha\left(\frac{1+|f|^2}{2},\frac{1-|f|^2}{2},f\right)
+\left(f\cdot n_\alpha,-f\cdot n_\alpha,n_\alpha\right).
\]
Then $\{\xi_{1},\cdots,\xi_p\}$ be the orthonormal basis of the
orthogonal complement of $\mathrm{Span}\{Y,N,Y_j|1\le j\le m\}$.
And $\{Y,N,Y_j,\xi_{\alpha}\}$ form a moving frame in $R^{m+p+2}_1$
along $M^m$.
\begin{Remark}
Geometrically, $\xi_\alpha$ corresponds to the unique sphere tangent to $M^m$ at one point $x$ with normal vector $n_\alpha$ and the same mean curvature $H^\alpha(x)$. We call $\{\xi_\alpha\}$ the mean curvature spheres of $M^m$.
\end{Remark}
We will use the following range of indices in this section: $1\leq
i,j,k\leq m; 1\leq \alpha,\beta\leq p$. We can write the structure equations
as below:
\begin{eqnarray*}
&&dY=\sum_i \omega_i Y_i,\\
&&dN=\sum_{ij}A_{ij}\omega_i Y_j+\sum_{i,\alpha} C^{\alpha}_i\omega_i \xi_{\alpha},\\
&&d Y_i=-\sum_j A_{ij}\omega_j Y-\omega_i N+\sum_j\omega_{ij}Y_j
+\sum_{j,\alpha} B^{\alpha}_{ij}\omega_j \xi_{\alpha},\\
&&d \xi_{\alpha}=-\sum_i C^{\alpha}_i\omega_i Y-\sum_{ij}\omega_i
B^{\alpha}_{ij}Y_j +\sum_{\beta} \theta_{\alpha\beta}\xi_{\beta},
\end{eqnarray*}
where $\omega_{ij}$ are the connection $1$-forms of the M\"{o}bius
metric $g$ and $\theta_{\alpha\beta}$ the normal connection $1$-forms. The
tensors
\[
{\bf A}=\sum_{ij}A_{ij}\omega_i\otimes\omega_j,~~ {\bf
B}=\sum_{ij\alpha}B^{\alpha}_{ij}\omega_i\otimes\omega_j \xi_{\alpha},~~
\Phi=\sum_{j\alpha}C^{\alpha}_j\omega_j \xi_{\alpha}
\]
are called the Blaschke tensor, the M\"{o}bius second fundamental
form and the M\"{o}bius form of $x$, respectively. The covariant
derivatives of $C^{\alpha}_i, A_{ij}, B^{\alpha}_{ij}$ are defined by
\begin{eqnarray*}
&&\sum_j C^{\alpha}_{i,j}\omega_j=d C^{\alpha}_i+\sum_j C^{\alpha}_j\omega_{ji}
+\sum_{\beta} C^{\beta}_j\theta_{\beta\alpha},\\
&&\sum_k A_{ij,k}\omega_k=d A_{ij}+\sum_k A_{ik}\omega_{kj}+\sum_k A_{kj}\omega_{ki},\\
&&\sum_k B^{\alpha}_{ij,k}\omega_k=d B^{\alpha}_{ij}+\sum_k
B^{\alpha}_{ik}\omega_{kj} +\sum_k B^{\alpha}_{kj}\omega_{ki}+\sum_{\beta}
B^{\beta}_{ij}\theta_{\beta\alpha}.
\end{eqnarray*}
The integrability conditions for the structure equations are given by
\begin{eqnarray}
&&A_{ij,k}-A_{ik,j}=\sum_{\alpha}(B^{\alpha}_{ik}C^{\alpha}_j-B^{\alpha}_{ij}C^{\alpha}_k),\label{equa1}\\
&&C^{\alpha}_{i,j}-C^{\alpha}_{j,i}=\sum_k(B^{\alpha}_{ik}A_{kj}-B^{\alpha}_{jk}A_{ki}),\label{equa2}\\
&&B^{\alpha}_{ij,k}-B^{\alpha}_{ik,j}=\delta_{ij}C^{\alpha}_k-\delta_{ik}C^{\alpha}_j,\label{equa3}\\
&&R_{ijkl}=\sum_{\alpha}(B^{\alpha}_{ik}B^{\alpha}_{jl}-B^{\alpha}_{il}B^{\alpha}_{jk})
+\delta_{ik}A_{jl}+\delta_{jl}A_{ik}
-\delta_{il}A_{jk}-\delta_{jk}A_{il},\label{equa4}\\
&&R^{\perp}_{\alpha\beta ij}=\sum_k
(B^{\alpha}_{ik}B^{\beta}_{kj}-B^{\beta}_{ik}B^{\alpha}_{kj}). \label{equa5}
\end{eqnarray}
Here $R_{ijkl}$ denote the curvature tensor of $g$.
Other restrictions on tensors $\bf A, B$ are
\begin{eqnarray}
&&\sum_j B^{\alpha}_{jj}=0, ~~~\sum_{ijr}(B^{\alpha}_{ij})^2=\frac{m-1}{m}, \label{equa7}\\
&& tr{\bf A}=\sum_j A_{jj}=\frac{1}{2m}(1+m^2\kappa).\label{equa8}
\end{eqnarray}
Where $\kappa=\frac{1}{n(n-1)}\sum_{ij}R_{ijij}$ is its normalized
M\"{o}bius scalar curvature.
 We know that all coefficients in the
structure equations are determined by $\{g, {\bf B}\}$ and the
normal connection $\{\theta_{\alpha\beta}\}$. Coefficients of M\"{o}bius invariants and
the isometric invariants are also related by \cite{CPWang}
\begin{align}
B^{\alpha}_{ij}&=\rho^{-1}(h^{\alpha}_{ij}-H^{\alpha}\delta_{ij}),\label{2.22}\\
C^{\alpha}_i&=-\rho^{-2}[H^{\alpha}_{,i}+\sum_j(h^{\alpha}_{ij}-H^{\alpha}\delta_{ij})e_j(\ln\rho)]. \label{2.23}
\end{align}

\section{M\"{o}bius invariants on Wintgen ideal submanifolds}
A submanifold $f:M^{m}\to
\mathbb{R}^{m+p}$ is a Wintgen ideal submanifold if and only if, at
each point of $M^{m}$, there is a suitable frame such that the second
fundamental form has the form (\ref{form1}). If $\mu_0=0$ in (\ref{form1}), then the Wintgen ideal submanifold is totally umbilical submanifold.
Next we consider non-umbilical Wintgen ideal submanifolds, that is $\mu_0\neq 0$ on $M^m$ and $m\geq 3$.

Since $\mu_0\neq 0$, we can choose a local orthonormal basis $\{E_1,\cdots,E_m\}$ of $TM^m$ with respect to the M\"{o}bius metric $g$ and a local orthonormal basis $\{\xi_1,\cdots,\xi_p\}$ of $T^{\bot}M^m$, such that the coefficients of the M\"obius second fundamental form ${\bf
B}$
have the form
\begin{equation}\label{eq-equality}
B^{1}=
\begin{pmatrix}
0 & \mu & 0 & \cdots & 0\\
\mu & 0 & 0 & \cdots & 0\\
0  & 0 & 0 & \cdots & 0\\
\vdots & \vdots & \vdots & \ddots & \vdots\\
0  & 0 & 0 & \cdots & 0
\end{pmatrix},~~~
B^{2}=
\begin{pmatrix}
\mu & 0 & 0 & \cdots & 0\\
0 & -\mu & 0 & \cdots & 0\\
0  & 0 & 0 & \cdots & 0\\
\vdots & \vdots & \vdots & \ddots & \vdots\\
0  & 0 & 0 & \cdots & 0
\end{pmatrix};~~
B^{\alpha}=0,~~\alpha\ge 3.
\end{equation}
By \eqref{equa7}, the norm of ${\bf B}$ is constant and
$\mu=\sqrt{\frac{m-1}{4m}}$. Clearly the distribution $\mathbb{D}=span\{E_1,E_2\}$ is well-defined.
For convenience we adopt the convention
below on the range of indices:
\[
1\le i,j,k,l\le m,~~ 3\leq a,b,c\leq m,~~1\le \alpha,\beta,\gamma \le
p.
\]

First we compute the covariant derivatives of $B^{\alpha}_{ij}$.
Since the M\"obius second fundamental form ${\bf B}$ has the form (\ref{eq-equality}), using the definition of the covariant derivatives of $B^{\alpha}_{ij}$, we have
\begin{equation}\label{bb1}
\begin{split}
&B^{\delta}_{ab,k}=0,~~~ 1\leq \delta\leq p, 1\leq k\leq m;~~B^{\alpha}_{1a,i}=0,~~~  B^{\alpha}_{2a,i}=0,~~ \alpha\geq 3,\\
&\theta_{1\alpha}=\frac{B^{\alpha}_{12,1}}{\mu}\omega_1+\frac{B^{\alpha}_{12,2}}{\mu}\omega_2,~~~~
\theta_{2\alpha}=\frac{B^{\alpha}_{11,1}}{\mu}\omega_1+\frac{B^{\alpha}_{11,2}}{\mu}\omega_2, \alpha\geq 3.
\end{split}
\end{equation}
\begin{equation}\label{bb2}
\omega_{2a}=\sum_i\frac{B^1_{1a,i}}{\mu}\omega_i=-\sum_i\frac{B^2_{2a,i}}{\mu}\omega_i,~~ \omega_{1a}=\sum_i\frac{B^1_{2a,i}}{\mu}\omega_i=\sum_i\frac{B^2_{1a,i}}{\mu}\omega_i.
\end{equation}
\begin{equation}\label{bb3}
\begin{split}
&2\omega_{12}+\theta_{12}=\sum_i\frac{-B^1_{11,i}}{\mu}\omega_i=\sum_i\frac{B^1_{22,i}}{\mu}\omega_i=\sum_i\frac{B^2_{12,i}}{\mu}\omega_i,\\
&B^1_{12,i}=0,~~B^2_{11,i}=B^2_{22,i}=0.
\end{split}
\end{equation}

It follows from (\ref{equa3}) and (\ref{bb1})
that, when $\alpha\geq 3$,
$$ C^{\alpha}_1=B^{\alpha}_{aa,1}-B^{\alpha}_{a1,a}=0,~~
C^{\alpha}_2=B^{\alpha}_{aa,2}-B^{\alpha}_{a2,a}=0;$$
$$ C^{\alpha}_a=B^{\alpha}_{11,a}-B^{\alpha}_{1a,1}=B^{\alpha}_{11,a},
~~C^{\alpha}_a=B^{\alpha}_{22,a}-B^{\alpha}_{2a,2}=B^{\alpha}_{22,a}.
$$
Since $\sum_iB^{\delta}_{ii,k}=0, 1 \leq \delta \leq p,1\leq k\leq m $, we have
$$C_i^{\alpha}=0,~~\alpha\geq 3.$$
From (\ref{bb2}) and (\ref{bb3}), we obtain
$$
B^1_{2a,2}=B^1_{22,a}=B^2_{1a,2},~~
B^2_{1a,1}=0,~~B^2_{2a,2}=0.$$
This implies that $C^1_a=B^1_{22,a}-B^1_{2a,2}=0$. Similarly $C^2_a=0.$

The other coefficients of $\{C^{r}_j\}$ are obtained similarly as
below:
\begin{equation}\label{cc1}
\begin{split}
C^{1}_1=-B^{1}_{1a,a}=-\mu\omega_{2a}(e_a),
~~C^{2}_2=-B^{2}_{2a,a}=\mu\omega_{2a}(e_a), \\
C^{1}_2=-B^{1}_{2a,a}=-\mu\omega_{1a}(e_a), ~~C^{2}_1=-B^{2}_{1a,a}=-\mu\omega_{1a}(e_a).
\end{split}
\end{equation}
In particular we have
\begin{equation}
C^{1}_1=-C^{2}_2,~~C^{1}_2=C^{2}_1.
\end{equation}

\begin{lemma}\label{le1}
In the sub-bundles $Span\{E_1,E_2\}$ and $Span\{\xi_1,\xi_2\}$,
we can always choose new orthonormal basis $\{E_1,E_2\}$ and $\{\xi_1,\xi_2\}$ such that the M\"{o}bius second fundamental form ${\bf B}$ still takes the form \eqref{eq-equality}, and the coefficients of the M\"{o}bius form
satisfy $$C^1_1=-C^2_2,~~C^1_2=C^2_1=0, ~~ C^1_a=C^2_a=0,~~~C^{\alpha}_i=0,~~~ \alpha\geq 3.$$
\end{lemma}
\begin{proof} Under a new basis given as below:
\[
\left\{
\begin{array}{cc}
\tilde{E}_1 =& \cos\theta E_1+\sin\theta E_2,\\
\tilde{E}_2 =&-\sin\theta E_1+\cos\theta E_2,
\end{array}
\right.
\left\{
\begin{array}{cc}
\tilde{\xi}_1 =& \cos\varphi \xi_1+\sin\varphi \xi_2,\\
\tilde{\xi}_2 =& -\sin\varphi \xi_1+\cos\varphi \xi_2,
\end{array}
\right.
\]
we have
\[
\left(
\begin{array}{cc}
\tilde{B}^1_{11}&\tilde{B}^1_{12}\\
\tilde{B}^1_{21}&\tilde{B}^1_{22}
\end{array}
\right)=
\left(
\begin{array}{cc}
\sin(2\theta+\varphi)\mu & \cos(2\theta+\varphi)\mu\\
\cos(2\theta+\varphi)\mu & -\sin(2\theta+\varphi)\mu
\end{array}
\right),
\]
\[
\left(
\begin{array}{cc}
\tilde{B}^2_{11}&\tilde{B}^2_{12}\\
\tilde{B}^2_{21}&\tilde{B}^2_{22}
\end{array}
\right)=
\left(
\begin{array}{cc}
\cos(2\theta+\varphi)\mu & -\sin(2\theta+\varphi)\mu\\
-\sin(2\theta+\varphi)\mu & -\cos(2\theta+\varphi)\mu
\end{array}
\right),
\]
\[
\left(
\begin{array}{cc}
\tilde{C}^1_1&\tilde{C}^1_2\\
\tilde{C}^2_1&\tilde{C}^2_2
\end{array}
\right)=
\left(
\begin{array}{cc}
\cos(\theta+\varphi)C^1_1+\sin(\theta+\varphi)C^1_2 & \cos(\theta+\varphi)C^1_2-\sin(\theta+\varphi)C^1_1\\
\cos(\theta+\varphi)C^1_2-\sin(\theta+\varphi)C^1_1 & -\cos(\theta+\varphi)C^1_1-\sin(\theta+\varphi)C^1_2
\end{array}
\right).
\]
Let $\varphi=-2\theta$, then the coefficients of the M\"{o}bius second fundamental form ${\bf B}$ satisfy \eqref{eq-equality}. Clearly there exists a value $\theta$
such that $\tilde{C}^1_2=\tilde{C}^2_1=0.$
\end{proof}

\begin{lemma}\label{le2}
We can choose a local orthonormal basis basis $\{E_3,\cdots,E_m\}$ such that
$$B^1_{11,4}=B^1_{11,5}=\cdots=B^1_{11,m}=0.$$
\end{lemma}
\begin{proof} Let $E=\sum_{a=3}^mB^1_{11,a}E_a$. If $E=0,$, then the Lemma is true. If $E\neq 0,$ then we can choose a local orthonormal basis $\{\tilde{E}_3,\cdots,\tilde{E}_m\}$ in $Span\{E_3,\cdots,E_m\}$
such that $\tilde{E}_3=\frac{E}{|E|}$. Clearly, under this basis $B^1_{11,4}=\cdots=B^1_{11,m}=0$ as desired.
\end{proof}

From (\ref{bb1}), (\ref{bb2}), (\ref{bb3}), Lemma\ref{le1} and Lemma\ref{le2},  we write out the connection forms with respect to the local orthonormal basis $\{E_1,E_2,\cdots,E_m\}$:
\begin{equation}\label{conne0}
\begin{split}
&2\omega_{12}+\theta_{12}=\frac{C^1_1}{\mu}\omega_1-\frac{B^1_{11,3}}{\mu}\omega_3,\\
&\omega_{13}=-\frac{B^1_{11,3}}{\mu}\omega_2,~~~~\omega_{1i}=0, i\geq 4,\\
&\omega_{23}=\frac{B^1_{11,3}}{\mu}\omega_1-\frac{C^1_1}{\mu}\omega_3, ~~\omega_{2i}=-\frac{C^1_1}{\mu}\omega_i, i\geq 4.
\end{split}
\end{equation}
Combining  $C^1_2=0$ and the definition of $C^{\alpha}_{i,j}$,
we have $$C^1_1(\omega_{12}+\theta_{12})=\sum_k C^1_{2,k}\omega_k.$$
Combining (\ref{conne0}), we obtain that
\begin{equation}\label{conne2}
C^1_1\omega_{12}=\frac{(C^1_1)^2}{\mu}\omega_1-\frac{C^1_1B^1_{11,3}}{\mu}\omega_3-\sum_kC^1_{2,k}\omega_k.
\end{equation}
Using $d\omega_{ij}-\sum_k\omega_{ik}\wedge\omega_{kj}=-\frac{1}{2}\sum_{kl}R_{ijkl}\omega_k\wedge\omega_l$, (\ref{conne0}) and (\ref{conne2}),
we have the following equations
\begin{equation}\label{curve1}
\begin{split}
\sum_{k<l}R_{13kl}\omega_k\wedge\omega_l&=\frac{dB^1_{11,3}}{\mu}\wedge\omega_2+\sum_k\frac{C^1_{2,k}}{\mu}\omega_k\wedge\omega_3\\
&+\frac{(B^1_{11,3})^2-(C^1_1)^2}{\mu^2}\omega_1\wedge\omega_3,\\
\sum_{k<l}R_{1ikl}\omega_k\wedge\omega_l
&=\frac{-(C^1_1)^2}{\mu^2}\omega_1\wedge\omega_i+\frac{C^1_1B^1_{11,3}}{\mu^2}\omega_3\wedge\omega_i\\
&+\sum_k\frac{C^1_{2,k}}{\mu}\omega_k\wedge\omega_i
-\frac{B^1_{11,3}}{\mu}\omega_2\wedge\omega_{3i}, ~~~i\geq 4,\\
\end{split}
\end{equation}
\begin{equation}\label{curve12}
\begin{split}
\sum_{k<l}R_{23kl}\omega_k\wedge\omega_l
&=\frac{-dB^1_{11,3}}{\mu}\wedge\omega_1+\sum_k\frac{C^1_{1,k}}{\mu}\omega_k\wedge\omega_3\\
&+\frac{(B^1_{11,3})^2-(C^1_1)^2}{\mu^2}\omega_2\wedge\omega_3
-\frac{2C^1_1B^1_{11,3}}{\mu^2}\omega_1\wedge\omega_2,\\
\sum_{k<l}R_{2ikl}\omega_k\wedge\omega_l
&=\frac{-(C^1_1)^2}{\mu^2}\omega_2\wedge\omega_i+\sum_k\frac{C^1_{1,k}}{\mu}\omega_k\wedge\omega_i\\
&+\frac{B^1_{11,3}}{\mu}\omega_1\wedge\omega_{3i},~~~~ i\geq 4.
\end{split}
\end{equation}
Using (\ref{equa4}), from (\ref{curve1}) and (\ref{curve12}), we obtain that the coefficients of the Blaschke tensor
satisfy
\begin{equation}\label{mcof1}
\begin{split}
A_{11}+A_{33}=\frac{C^1_{2,1}}{\mu}+\frac{(B^1_{11,3})^2
-(C^1_1)^2}{\mu^2};
~~A_{11}+A_{ii}=\frac{C^1_{2,1}}{\mu}-\frac{(C^1_1)^2}{\mu^2}, i\geq 4;\\
A_{22}+A_{33}=\frac{C^1_{1,2}}{\mu}+\frac{(B^1_{11,3})^2
-(C^1_1)^2}{\mu^2};
~~A_{22}+A_{ii}=\frac{C^1_{1,2}}{\mu}-\frac{(C^1_1)^2}{\mu^2}, i\geq 4.
\end{split}
\end{equation}
\begin{equation}\label{mcof2}
\begin{split}
&A_{23}=\frac{C^1_{1,3}}{\mu}=\frac{E_1(B^1_{11,3})}{\mu},\\
&A_{13}=\frac{2C^1_1B^1_{11,3}}{\mu^2}-\frac{E_2(B^1_{11,3})}{\mu}=\frac{C^1_{2,3}}{\mu}+\frac{C^1_1B^1_{11,3}}{\mu^2},\\
&A_{1i}=\frac{C^1_{2,i}}{\mu}=-\frac{B^1_{11,3}}{\mu}
\omega_{3i}(E_2), ~~
A_{2i}=\frac{C^1_{1,i}}{\mu}=\frac{B^1_{11,3}}{\mu}
\omega_{3i}(E_1), i\geq 4,
\end{split}
\end{equation}
\begin{equation}\label{mcof3}
\begin{split}
&A_{ij}=0, ~~~i,j\geq 4,~~ i\neq j,\\
&A_{12}=\frac{C^1_{2,2}}{\mu}-\frac{E_3(B^1_{11,3})}{\mu}=\frac{C^1_{1,1}}{\mu}+\frac{E_3(B^1_{11,3})}{\mu}.
\end{split}
\end{equation}
\begin{equation}\label{mcof4}
\begin{split}
&E_i(B^1_{11,3})=0, ~~~E_3(B^1_{11,3})=B^1_{11,3}\omega_{3i}(E_i),~~ i\geq 4,\\
&\omega_{3i}(E_3)=0, i\geq 4, ~~~\omega_{3i}(E_j)=0,~~ i,j\geq 4, i\neq j.
\end{split}
\end{equation}
From (\ref{mcof1}), we obtain the coefficients of the Blaschke tensor satisfy
$$A_{ij}=A_4\delta_{ij},~~i,j\geq 4,~~A_4\triangleq A_{44}.$$

\section{Wintgen ideal submanifolds constructed by cones}

\begin{Definition}
Let $u:M^r\longrightarrow \mathbb{S}^{r+p}\subset \mathbb{R}^{r+p+1}$ be an immersed submanifold.
We define \emph{the cone over $u$} in $\mathbb{R}^{m+p}$ as
\begin{equation*}
\begin{split}
&f:R^+\times \mathbb{R}^{m-r-1}\times M^r\longrightarrow \mathbb{R}^{m+p},\\
&~~~~~~f(t,y,u)=(y,tu),
\end{split}
\end{equation*}
\end{Definition}

\begin{PROPOSITION}\label{32}
Let $u: M^r\longrightarrow \mathbb{S}^{r+p}$ be an immersed submanifold. Then the cone
$f=(y,tu):R^+\times \mathbb{R}^{m-r-1}\times M^r\longrightarrow \mathbb{R}^{m+p}$ is a Wintgen ideal submanifold if and only if $u$ is a minimal Wintgen ideal submanifold in $\mathbb{S}^{r+p}$.
\end{PROPOSITION}
\begin{proof}
The first and second
fundamental forms of $f$ are, respectively,
\begin{equation}\label{hi1}
I=t^2I_u+I_{\mathbb{R}^{m-r}}, \;\; II=t~II_u,
\end{equation}
where $I_u,II_u$ are the first and second fundamental forms of $u$, respectively, and $I_{\mathbb{R}^{m-r}}$ denotes the standard metric of $\mathbb{R}^{m-r}$.
The conclusion follows easily.
\end{proof}

The M\"{o}bius position vector $Y:R^+\times \mathbb{R}^{m-r-1}\times M^r\longrightarrow \mathbb{R}^{m+p+2}_1$ of the cone $f$ is
\begin{equation*}
Y=\rho_0\left(\frac{1+t^2+|y|^2}{2t},\frac{1-t^2-|y|^2}{2t},y,u\right),
\end{equation*}
where $\rho_0^2=\frac{m}{m-1}(|II_u|^2-mH_u^2): M^r\longrightarrow \mathbb{R}$, and $y:\mathbb{R}^{m-r-1}\longrightarrow \mathbb{R}^{m-r-1}$ is the identity map.
Let $$\mathbb{H}^{m-r}=\{(y_0,y)\in \mathbb{R}^{m-r+1}:~-y_0^2+|y|^2=-1, y_0\geq 1\}\cong R^+\times \mathbb{R}^{m-r-1},$$
then $(\frac{1+t^2+|y|^2}{2t},\frac{1-t^2-|y|^2}{2t},y):R^+\times \mathbb{R}^{m-r-1}\cong\mathbb{H}^{m-r}\to \mathbb{H}^{m-r}$ is nothing else but the identity map. And the M\"{o}bius position vector of the cone $f$ is
\begin{equation}\label{mop2}
Y=\rho_0(id,u):\mathbb{H}^{m-r}\times M^r\to \mathbb{H}^{m-r}\times \mathbb{S}^{r+p}\subset \mathbb{R}^{m+p+2}_1,
\end{equation}
where $\rho_0\in C^{\infty}(M^r)$ and $id:\mathbb{H}^{m-r}\to \mathbb{H}^{m-r}$ is a identity map.

The m\"{o}bius metric of the cone $f$ is
$$g=\rho_0^2(I_u+I_{\mathbb{H}^{m-r}}),$$
where $I_{\mathbb{H}^{m-r}}$ is the standard hyperbolic metric of $\mathbb{H}^{m-r}$.

From (\ref{mop2}) we have the following result.
\begin{PROPOSITION}\label{redu}
Let $f:M^m\to \mathbb{R}^{m+p}$ be
an immersed submanifold without umbilical points. If there exists a submanifold $u:M^r\to \mathbb{S}^{r+p}$ such that the M\"{o}bius position vector of $f$ is
\begin{equation*}
Y=\rho_0(id,u):\mathbb{H}^{m-r}\times M^r\to \mathbb{H}^{m-r}\times \mathbb{S}^{r+p}\subset \mathbb{R}^{m+p+2}_1,
\end{equation*}
where $\rho_0\in C^{\infty}(M^r)$ and $id:\mathbb{H}^{m-r}\to \mathbb{H}^{m-r}$ is the identity map.
Then $f$ is a cone over $u$.
\end{PROPOSITION}
By computation, and combining with (\ref{2.23}), (\ref{hi1}), we have the following result.
\begin{PROPOSITION}\label{redu1}
Let the cone
$f=(y,tu):R^+\times \mathbb{R}^{m-r-1}\times M^r\longrightarrow \mathbb{R}^{m+p}$ be a Wintgen ideal submanifold. Then
the M\"{o}bius form $\Phi$ of $f$ vanishes if and only if the M\"{o}bius form $\Phi_u$ of $u:M^r\to S^{r+p}$ vanishes.
\end{PROPOSITION}

\section{The vanishing of the M\"{o}bius form}

A commnon feature of all the examples of M\"obius homogeneous, Wintgen ideal submanifolds described in the introduction is that they all have vanishing M\"obius form, i.e., $\Phi=0$.
For Examples~\ref{ex1} and \ref{ex2}, this follows from Proposition~\ref{redu1} and the results in \cite{li}. For Example~\ref{ex3} in $S^5$, this has been verified in \cite{XLMW}.

Conversely, in this section we show that any M\"obius homogeneous, Wintgen ideal submanifold must have this property. This will be shown by contradiction.

Note that starting from this section, the assumption of M\"obius homogeneity will be used, whose basic consequence is that any M\"obius invariant geometric quantity as a function well-defined at every point of the underlying manifold must be a constant.

\begin{lemma}\label{le3}
Let $f: M^m\to \mathbb{R}^{m+p}, (m\geq 3)$ be a M\"{o}bius homogeneous Wintgen ideal submanifold. If $\Phi\neq 0$, then  $B^1_{11,3}=0$.
\end{lemma}
\begin{proof}
By assumption $\Phi\neq0$, i.e., $C^1_1\neq0$. By Lemma~\ref{le1}, the normal vector field $\xi_1$ and tangent vector fields $\{E_1,E_2\}$ are determined up to a sign.
Since $f$ is M\"{o}bius homogeneous, the function $C^1_1$ takes value as a constant.
Similarly, if $B^1_{11,3}\neq0$, then the function $B^1_{11,3}$ is well-defined function, hence a non-zero constant.

From
$dC^1_1=\sum_kC^1_{1,k}\omega_k$, we have
$$C^1_{1,k}=0, ~~1\leq k\leq m.$$
From (\ref{mcof2}), (\ref{mcof3}) and (\ref{mcof4}), we obtain
\begin{equation}\label{abc}
A_{21}=A_{23}=A_{24}=\cdots=A_{2m}=0,~~ A_{13}=\frac{2C^1_1B^1_{11,3}}{\mu^2},~~C^1_{2,3}=\frac{C^1_1B^1_{11,3}}{\mu}.
\end{equation}

If $B^1_{11,3}\neq0$, from (\ref{mcof2}) and (\ref{mcof4}), we have
$$\omega_{3i}=-\frac{\mu A_{1i}}{B^1_{11,3}}\omega_2, ~~i\geq 4.$$
Using $d\omega_{ij}-\sum_k\omega_{ik}\wedge\omega_{kj}
=-\frac{1}{2}\sum_{kl}R_{ijkl}\omega_k\wedge\omega_l$,
there follows
\begin{equation}\label{3iir}
\begin{split}
&\sum_{k<l}R_{3ikl}\omega_k\wedge\omega_l=d(\frac{\mu A_{1i}}{B^1_{11,3}})\wedge\omega_2-\omega_2\wedge(\sum_{k\geq4}\frac{\mu A_{1m}}{B^1_{11,3}}\omega_{mi})
+\frac{C^1_1B^1_{11,3}}{\mu^2}\omega_1\wedge\omega_i\\
&-\frac{\mu A_{1i}}{B^1_{11,3}}[\frac{C^1_{2,2}}{C^1_1}\omega_1\wedge\omega_{2}+\frac{C^1_{2,3}}{C^1_1}\omega_1\wedge\omega_3
+\omega_1\wedge(\sum_{k\geq4}\frac{C^1_{2,k}}{C^1_1}\omega_k)]
-\frac{(C^1_1)^2}{\mu^2}\omega_3\wedge\omega_i.
\end{split}
\end{equation}
Comparing $\omega_1\wedge\omega_i$ in (\ref{3iir}), there should be
$$R_{3i1i}=-\frac{\mu A_{1i}}{B^1_{11,3}}\frac{C^1_{2,i}}{C^1_1}+\frac{C^1_1B^1_{11,3}}{\mu^2}.$$
Since $A_{13}=\frac{2C^1_1B^1_{11,3}}{\mu^2}$ and $C^1_{2,i}=\mu A_{1i}$, we obtain
$$\frac{(C^1_1)^2(B^1_{11,3})^2}{\mu^4}+(A_{1i})^2=0,$$
which is a contradiction. So $B^1_{11,3}=0$.
\end{proof}
\begin{lemma}\label{le30}
Let $f: M^m\to \mathbb{R}^{m+p}, (m\geq 3)$ be a M\"{o}bius homogeneous Wintgen ideal submanifold. If $\Phi\neq 0$, then  the coefficients of the Blaschke tensor
satisfy
\begin{equation*}
\begin{split}
&A_{22}=A_{33}=\cdots=A_{mm}=-\frac{(C^1_1)^2}{2\mu^2},\\
&(A_{11}-A_{22})^2=\frac{(C^1_1)^2}{\mu^2}(A_{11}-A_{22})
+2(C^1_1)^2,
\end{split}
\end{equation*}
and curvature tensor satisfies
$$4\mu^2=R_{1212}+\sum_{\alpha\geq 3}\frac{(B^{\alpha}_{11,2})^2+(B^{\alpha}_{22,1})^2}{\mu^2}~.$$
\end{lemma}
\begin{proof}
From (\ref{mcof3}), (\ref{mcof4}),(\ref{abc}) and Lemma~\ref{le3}, we obtain
\begin{equation}\label{abc1}
\begin{split}
&A_{12}=A_{13}=\cdots=A_{1m}=0, ~~C^1_{2,2}=C^1_{2,3}=\cdots=C^1_{2,m}=0,\\
&\omega_{12}=(\frac{C^1_1}{\mu}-\frac{C^1_{2,1}}{C^1_1})\omega_1,~~~\omega_{1i}=0,~~~\omega_{2i}=-\frac{C^1_1}{\mu}\omega_i, ~~i\geq3.
\end{split}
\end{equation}
Noticing that local functions $A_{ii} (1\leq i\leq m)$ are constant, from the definition of $A_{ij,k}$, (\ref{abc}), and (\ref{abc1}), we get
\begin{equation*}
\begin{split}
&A_{12,1}=(A_{11}-A_{22})(\frac{C^1_1}{\mu}-\frac{C^1_{2,1}}{C^1_1}),~~A_{11,2}=0,\\
&A_{23,3}=(A_{33}-A_{22})\frac{C^1_1}{\mu},~~~~A_{33,2}=0.
\end{split}
\end{equation*}
Using (\ref{equa1}) and (\ref{mcof1}), we have
\begin{equation}\label{abc2}
\begin{split}
&A_{22}=A_{33}=\cdots=A_{mm}=-\frac{(C^1_1)^2}{2\mu^2},\\
&(A_{11}-A_{22})^2=\frac{(C^1_1)^2}{\mu^2}(A_{11}-A_{22})+2(C^1_1)^2.
\end{split}
\end{equation}
Using $d\omega_{12}-\sum_k\omega_{1k}\wedge\omega_{k2}=-\frac{1}{2}\sum_{kl}R_{12kl}\omega_k\wedge\omega_l$ and (\ref{abc1}),
we have
\begin{equation}\label{r1212}
R_{1212}=-(\frac{C^1_1}{\mu}-\frac{C^1_{2,1}}{C^1_1})^2.
\end{equation}
From (\ref{conne2}) and (\ref{abc1}), we get
\begin{equation}\label{noru}
\theta_{12}=(\frac{2C^1_{2,1}}{C^1_1}-\frac{C^1_1}{\mu})\omega_1,~~~~d\omega_1=(\frac{C^1_1}{\mu}-\frac{C^1_{2,1}}{C^1_1})\omega_1\wedge\omega_2,~~~~d\omega_2=0.
\end{equation}
Using $d\theta_{12}-\sum_{\tau}\theta_{1\tau}\wedge\theta_{\tau2}=
-\frac{1}{2}\sum_{kl}R^{\bot}_{12 kl}\omega_k\wedge\omega_l$ and (\ref{bb1}),
we have
\begin{equation}\label{nor12}
4\mu^2=R_{1212}+\sum_{\alpha\geq 3}\frac{(B^{\alpha}_{11,2})^2+(B^{\alpha}_{22,1})^2}{\mu^2}.
\qedhere
\end{equation}
\end{proof}
\begin{lemma}\label{le4}
Let $f: M^m\to\mathbb{ R}^{m+p}, (m\geq 3)$ be a M\"{o}bius homogeneous Wintgen ideal submanifold. If  $\Phi\neq 0$ and $p\geq 3$, then
we can choose orthonormal frames in $Span\{\xi_3,\cdots,\xi_p\}$ such that the normal connection have the following form
\[
\left\{
\begin{array}{ll}
\left\{
\begin{array}{cc}
\theta_{13}=a_0\omega_2,\\
\theta_{23}=-a_0\omega_1,\\
\end{array}
\right.
\left\{
\begin{array}{cc}
\theta_{14}=a_0\omega_1,\\
\theta_{24}=a_0\omega_2,\\
\end{array}
\right.
\left\{
\begin{array}{cc}
\theta_{1\alpha}=0,~~\alpha\geq 5,\\
\theta_{2\alpha}=0,~~\alpha\geq 5,\\
\end{array}
\right.\\
\theta_{34}=(3\frac{C^1_{2,1}}{C^1_1}-2\frac{C^1_1}{\mu})\omega_1,~~~a_0\neq0.\\
\end{array}
\right.
\]
\[
\left\{
\begin{array}{ll}
\left\{
\begin{array}{cc}
\theta_{(2k+1)(2k+3)}=a_k\omega_2,\\
\theta_{(2k+2)(2k+3)}=-a_k\omega_1,\\
\end{array}
\right.
\left\{
\begin{array}{cc}
\theta_{(2k+1)(2k+4)}=a_k\omega_1,\\
\theta_{(2k+2)(2k+4)}=a_k\omega_2,\\
\end{array}
\right.
\left\{
\begin{array}{cc}
\theta_{(2k+1)\alpha}=0,~~\alpha\geq 2k+5,\\
\theta_{(2k+2)\alpha}=0,~~\alpha\geq 2k+5,\\
\end{array}
\right.\\
\theta_{(2k+3)(2k+4)}=[(k+3)\frac{C^1_{2,1}}{C^1_1}-(k+2)\frac{C^1_1}{\mu}]\omega_1,~~~a_k\neq0.
\end{array}
\right.
\]
$$
\left\{
\begin{array}{cc}
\theta_{(2k+3)\alpha}=0,~~\alpha\geq 2k+5,\\
\theta_{(2k+4)\alpha}=0,~~\alpha\geq 2k+5,\\
\end{array}
\right.
$$
\end{lemma}
\begin{proof}
Since $p\geq 3$, from (\ref{bb1}), we can rechoose an orthonormal basis in $Span\{\xi_3,\cdots,\xi_p\}$ such that
\begin{equation}\label{nor13}
\left\{
\begin{array}{cc}
\theta_{13}=\frac{B^3_{11,2}}{\mu}\omega_1+\frac{B^3_{22,1}}{\mu}\omega_2,\\
\theta_{23}=-\frac{B^3_{22,1}}{\mu}\omega_1+\frac{B^3_{11,2}}{\mu}\omega_2,\\
\end{array}
\right.
\left\{
\begin{array}{cc}
\theta_{14}=\frac{B^4_{11,2}}{\mu}\omega_1,\\
\theta_{24}=\frac{B^4_{11,2}}{\mu}\omega_2,\\
\end{array}
\right.
\left\{
\begin{array}{cc}
\theta_{1\alpha}=0,~~\alpha\geq 5,\\
\theta_{2\alpha}=0,~~\alpha\geq 5,\\
\end{array}
\right.
\end{equation}
In fact, if $E=\sum_{\alpha\geq 3}B^{\alpha}_{22,1}\xi_{\alpha}\neq0$, let $\tilde{\xi}_3=\frac{E}{|E|}$. If $E=0$, let $\tilde{\xi}_3=\xi_3$.
If  $\tilde{E}=\sum_{\alpha\geq 4}B^{\alpha}_{11,2}\xi_{\alpha}\neq0$, let $\tilde{\xi}_4=\frac{\tilde{E}}{|\tilde{E}|}$. If $\tilde{E}=0$, let $\tilde{\xi}_4=\xi_4$. Therefore under the orthonormal
basis $\{\tilde{\xi}_3,\tilde{\xi}_4,\cdots,\tilde{\xi}_p\}$, we have (\ref{nor13}).

Since $f$ is M\"{o}bius homogeneous and the normal vector fields $\xi_3,\xi_4$ are determined up to a sign, then $B^3_{11,2}$, $B^3_{22,1}$ and $B^4_{11,2}$ are constants
(otherwise they are equal to zero). Noticing $R^{\bot}_{\alpha\beta kl}=0$ for $\alpha=1,2,\beta=3,4$, and using $d\theta_{\alpha\beta}-\sum_{\tau}\theta_{\alpha\tau}\wedge\theta_{\tau\beta}=
-\frac{1}{2}\sum_{kl}R^{\bot}_{\alpha\beta kl}\omega_k\wedge\omega_l$ and (\ref{nor13}), we obtain
\begin{equation}\label{nor132}
\begin{split}
&\frac{B^3_{11,2}}{\mu}(3\frac{C^1_{2,1}}{C^1_1}-2\frac{C^1_1}{\mu})\omega_1\wedge\omega_2=\frac{B^4_{11,2}}{\mu}\omega_1\wedge\theta_{34},\\
&\frac{B^3_{22,1}}{\mu}(3\frac{C^1_{2,1}}{C^1_1}-2\frac{C^1_1}{\mu})\omega_1\wedge\omega_2=-\frac{B^4_{11,2}}{\mu}\omega_2\wedge\theta_{34},\\
&(-\frac{B^3_{22,1}}{\mu}\omega_1+\frac{B^3_{11,2}}{\mu}\omega_2)\wedge\theta_{34}=0,\\
&\frac{B^4_{11,2}}{\mu}(3\frac{C^1_{2,1}}{C^1_1}-2\frac{C^1_1}{\mu})\omega_1\wedge\omega_2
=-(\frac{B^3_{11,2}}{\mu}\omega_1+\frac{B^3_{22,1}}{\mu}\omega_2)\wedge\theta_{34}.
\end{split}
\end{equation}
We eliminate the term $\theta_{34}$ in (\ref{nor132}) to obtain
\begin{equation}\label{norm1}
\frac{B^3_{22,1}}{\mu}\frac{B^3_{11,2}}{\mu}
\left(3\frac{C^1_{2,1}}{C^1_1}-2\frac{C^1_1}{\mu}\right)=0.
\end{equation}

If $2\frac{C^1_1}{\mu}=3\frac{C^1_{2,1}}{C^1_1}$, i.e., $\frac{C^1_{2,1}}{\mu}=\frac{2(C^1_1)^2}{3\mu^2}$,
noting $\frac{C^1_{2,1}}{\mu}=A_{11}-A_{22}$, then from (\ref{abc2}) we have
$$\frac{4}{9}\frac{(C^1_1)^4}{\mu^4}=\frac{2(C^1_1)^4}{3\mu^4}+2(C^1_1)^2,$$
which implies $-\frac{2}{9}\frac{(C^1_1)^4}{\mu^4}=2(C^1_1)^2$, and is a contradiction. So $2\frac{C^1_1}{\mu}-3\frac{C^1_{2,1}}{C^1_1}\neq0$.

If $B^3_{22,1}=0$, from (\ref{nor132}), we have
\begin{equation*}
\left\{
\begin{array}{cc}
\frac{B^3_{11,2}}{\mu}(3\frac{C^1_{2,1}}{C^1_1}-2\frac{C^1_1}{\mu})\omega_1\wedge\omega_2-\frac{B^4_{11,2}}{\mu}\omega_1\wedge\theta_{34}=0,\\
\frac{B^4_{11,2}}{\mu}(3\frac{C^1_{2,1}}{C^1_1}-2\frac{C^1_1}{\mu})\omega_1\wedge\omega_2+\frac{B^3_{11,2}}{\mu}\omega_1\wedge\theta_{34}=0,
\end{array}
\right.
\end{equation*}
which imply that $\frac{B^3_{11,2}}{\mu}=\frac{B^4_{11,2}}{\mu}=0$. Thus $\sum_{\alpha\geq 3}\frac{(B^{\alpha}_{11,2})^2+(B^{\alpha}_{22,1})^2}{\mu^2}=0$, and
equation (\ref{nor12}) implies $4\mu^2=R_{1212}$. Since $R_{1212}=-(\frac{C^1_1}{\mu}-\frac{C^1_{2,1}}{C^1_1})^2$, so this is a contradiction. Thus $B^3_{22,1}\neq0$.

Thus $B^3_{11,2}=0$. From (\ref{nor13}) and (\ref{nor132}) we have
\begin{equation*}
\left\{
\begin{array}{cc}
\frac{B^3_{22,1}}{\mu}(3\frac{C^1_{2,1}}{C^1_1}-2\frac{C^1_1}{\mu})\omega_1\wedge\omega_2+\frac{B^4_{11,2}}{\mu}\omega_2\wedge\theta_{34}=0,\\
\frac{B^4_{11,2}}{\mu}(3\frac{C^1_{2,1}}{C^1_1}-2\frac{C^1_1}{\mu})\omega_1\wedge\omega_2+\frac{B^3_{22,1}}{\mu}\omega_2\wedge\theta_{34}=0,
\end{array}
\right.
\end{equation*}
which imply $\frac{B^3_{22,1}}{\mu}=\pm\frac{B^4_{11,2}}{\mu}.$ We assume $a_0\triangleq\frac{B^3_{22,1}}{\mu}=\frac{B^4_{11,2}}{\mu}\neq 0,$ otherwise let $\tilde{\xi}_4=-\xi_4$.
From (\ref{nor13}) and (\ref{nor132}), we have
\begin{equation}\label{nor1301}
\left\{
\begin{array}{cc}
\theta_{13}=a_0\omega_2,\\
\theta_{23}=-a_0\omega_1,\\
\end{array}
\right.
\left\{
\begin{array}{cc}
\theta_{14}=a_0\omega_1,\\
\theta_{24}=a_0\omega_2,\\
\end{array}
\right.
\left\{
\begin{array}{cc}
\theta_{1\alpha}=0,~~\alpha\geq 5,\\
\theta_{2\alpha}=0,~~\alpha\geq 5,\\
\end{array}
\right.
\end{equation}
and
\begin{equation}\label{nor34}
\theta_{34}=\left(3\frac{C^1_{2,1}}{C^1_1}-2\frac{C^1_1}{\mu}\right)
\omega_1.
\end{equation}
Using $d\theta_{1\alpha}-\sum_{\tau}\theta_{1\tau}\wedge\theta_{\tau\alpha}=
-\frac{1}{2}\sum_{kl}R^{\bot}_{1\alpha kl}\omega_k\wedge\omega_l$ and (\ref{nor1301}), noting $R^{\bot}_{1\alpha kl}=0, \alpha\geq5,$
we have
\begin{equation*}
-\omega_1\wedge\theta_{3\alpha}+\omega_2\wedge\theta_{4\alpha}=0,~~~\omega_2\wedge\theta_{3\alpha}+\omega_1\wedge\theta_{4\alpha}=0,~~~\alpha\geq5.
\end{equation*}
Thus we can assume that
\begin{equation}\label{341}
\left\{
\begin{array}{cc}
\theta_{3\alpha}=a^{\alpha}_1\omega_1+a^{\alpha}_2\omega_2,\\
\theta_{4\alpha}=-a^{\alpha}_2\omega_1+a^{\alpha}_1\omega_2,\\
\end{array}
\right.
\alpha\geq5.
\end{equation}
We can choose a new orthonormal frame locally in $Span\{\xi_5,\cdots,\xi_p\}$ such that
\begin{equation}\label{nor35}
\left\{
\begin{array}{cc}
\theta_{35}=a\omega_1+b\omega_2,\\
\theta_{45}=-b\omega_1+a\omega_2,\\
\end{array}
\right.
\left\{
\begin{array}{cc}
\theta_{36}=c\omega_1,\\
\theta_{46}=c\omega_2,\\
\end{array}
\right.
\left\{
\begin{array}{cc}
\theta_{3\alpha}=0,~~\alpha\geq 7,\\
\theta_{4\alpha}=0,~~\alpha\geq 7,\\
\end{array}
\right.
\end{equation}

Using $d\theta_{35}=\sum_{\tau}\theta_{3\tau}\wedge\theta_{\tau5}$ and (\ref{nor34}), we have
\begin{equation}\label{nor56}
a\left(4\frac{C^1_{2,1}}{C^1_1}-3\frac{C^1_1}{\mu}\right)\omega_1\wedge\omega_2=c\omega_1\wedge\theta_{56}.
\end{equation}
Similarly, we have
\begin{equation}\label{nor3412}
\begin{split}
&b\left(4\frac{C^1_{2,1}}{C^1_1}-3\frac{C^1_1}{\mu}\right)\omega_1\wedge\omega_2=-c\omega_2\wedge\theta_{56},\\
&c\left(4\frac{C^1_{2,1}}{C^1_1}-3\frac{C^1_1}{\mu}\right)\omega_1\wedge\omega_2=-(a\omega_1+b\omega_2)\wedge\theta_{56},\\
&-b\omega_1\wedge\theta_{56}+a\omega_2\wedge\theta_{56}=0.
\end{split}
\end{equation}
Using the equations (\ref{nor56}) and (\ref{nor3412}), we can obtain
\begin{equation}\label{norabc}
ab\left(4\frac{C^1_{2,1}}{C^1_1}-3\frac{C^1_1}{\mu}\right)=0.
\end{equation}
If $4\frac{C^1_{2,1}}{C^1_1}-3\frac{C^1_1}{\mu}=0$, i.e., $\frac{C^1_{2,1}}{\mu}=\frac{3(C^1_1)^2}{4\mu^2}$,
noting $\frac{C^1_{2,1}}{\mu}=A_{11}-A_{22}$, then from (\ref{abc2}) we have
$$\frac{9}{16}\frac{(C^1_1)^4}{\mu^4}=\frac{3(C^1_1)^4}{4\mu^4}+2(C^1_1)^2,$$
which is a contradiction. So $4\frac{C^1_{2,1}}{C^1_1}-3\frac{C^1_1}{\mu}\neq0$.

If $b=0,$ from (\ref{nor56}) and (\ref{nor3412}), we have
\begin{equation}\label{nor34123}
\begin{split}
a\left(4\frac{C^1_{2,1}}{C^1_1}-3\frac{C^1_1}{\mu}\right)\omega_1\wedge\omega_2-c\omega_1\wedge\theta_{56}=0,\\
c\left(4\frac{C^1_{2,1}}{C^1_1}-3\frac{C^1_1}{\mu}\right)\omega_1\wedge\omega_2+a\omega_1\wedge\theta_{56}=0.
\end{split}
\end{equation}
Which implies $a=c=0$.
From (\ref{nor1301}) and (\ref{nor35}), we have
\begin{equation}\label{3412}
d\theta_{34}=\sum_{\tau}\theta_{3\tau}\wedge\theta_{\tau4}=\theta_{31}\wedge\theta_{14}+\theta_{32}\wedge\theta_{24}=
2a_0^2\omega_1\wedge\omega_2.
\end{equation}
On the other hand, from (\ref{nor34}), we have
$$d\theta_{34}=\left(3\frac{C^1_{2,1}}{C^1_1}-2\frac{C^1_1}{\mu}\right)
\left(\frac{C^1_1}{\mu}-\frac{C^1_{2,1}}{C^1_1}\right)\omega_1\wedge\omega_2.$$
Thus we have
\begin{equation}\label{3434}
\left(3\frac{C^1_{2,1}}{C^1_1}-2\frac{C^1_1}{\mu}\right)
\left(\frac{C^1_1}{\mu}-\frac{C^1_{2,1}}{C^1_1}\right)=
2a_0^2.
\end{equation}
Combining with $C^1_{2,1}=\mu(A_{11}-A_{22})$ and $A_{22}=-\frac{(C^1_1)^2}{2\mu^2}$, we have
\begin{equation}\label{a01}
2a_0^2=2A_{11}-\frac{(C^1_1)^2}{\mu^2}-6\mu^2.
\end{equation}
From (\ref{r1212}), we have
$$R_{1212}<0,~~i.e.,A_{11}<2\mu^2+\frac{C^1_1}{2\mu^2}.$$
Thus
$$2A_{11}-\frac{(C^1_1)^2}{\mu^2}-6\mu^2<-2\mu^2.$$
which is in contradiction with (\ref{a01}). Thus $b\neq0$.

Therefore $a=0$, From (\ref{nor3412}), we have
\begin{equation}\label{nor345}
\begin{split}
&b\left(4\frac{C^1_{2,1}}{C^1_1}-3\frac{C^1_1}{\mu}\right)\omega_1\wedge\omega_2=-c\omega_2\wedge\theta_{56},\\
&c\left(4\frac{C^1_{2,1}}{C^1_1}-3\frac{C^1_1}{\mu}\right)\omega_1\wedge\omega_2=-b\omega_2\wedge\theta_{56},\\
&\omega_1\wedge\theta_{56}=0,
\end{split}
\end{equation}
which implies $b=\pm c$, and we can assume that $b=c=a_1$. Thus
\begin{equation}\label{nor3451}
\left\{
\begin{array}{cc}
\theta_{35}=a_1\omega_2,\\
\theta_{45}=-a_1\omega_1,\\
\end{array}
\right.
\left\{
\begin{array}{cc}
\theta_{36}=a_1\omega_1,\\
\theta_{46}=a_1\omega_2,\\
\end{array}
\right.
\left\{
\begin{array}{cc}
\theta_{3\alpha}=0,~~\alpha\geq 7,\\
\theta_{4\alpha}=0,~~\alpha\geq 7,\\
\end{array}
\right.
\end{equation}
and
\begin{equation}\label{nor34511}
\theta_{56}=\left(4\frac{C^1_{2,1}}{C^1_1}-3\frac{C^1_1}{\mu}
\right)\omega_1.
\end{equation}
Repeating the process (\ref{nor34})--(\ref{nor34511}), we have
\[
\left\{
\begin{array}{cc}
\left\{
\begin{array}{cc}
\theta_{(2k+1)(2k+3)}=a_k\omega_2,\\
\theta_{(2k+2)(2k+3)}=-a_k\omega_1,\\
\end{array}
\right.
\left\{
\begin{array}{cc}
\theta_{(2k+1)(2k+4)}=a_k\omega_1,\\
\theta_{(2k+2)(2k+4)}=a_k\omega_2,\\
\end{array}
\right.
\left\{
\begin{array}{cc}
\theta_{(2k+1)\alpha}=0,~~\alpha\geq 2k+5,\\
\theta_{(2k+2)\alpha}=0,~~\alpha\geq 2k+5,\\
\end{array}
\right.\\
\theta_{(2k+3)(2k+4)}=[(k+3)\frac{C^1_{2,1}}{C^1_1}-(k+2)\frac{C^1_1}{\mu}]\omega_1, ~~a_k\neq0.
\end{array}
\right.
\]
\[
\left\{
\begin{array}{cc}
\theta_{(2k+3)\alpha}=0,~~\alpha\geq 2k+5,\\
\theta_{(2k+4)\alpha}=0,~~\alpha\geq 2k+5,\\
\end{array}
\right.
\]
Thus we finish the proof of Lemma\ref{le4}.
\end{proof}
\begin{Remark}
During the proof of Lemma\ref{le4}, we can assume that the codimension of $f$ is sufficiently large, otherwise
we consider $f: M^m\to \mathbb{R}^{m+p}\hookrightarrow \mathbb{R}^{m+p+k}$ as a submanifold in $\mathbb{R}^{m+p+k}$, which also is a M\"{o}bius homogeneous Wintgen ideal submanifold
in $\mathbb{R}^{m+p+k}$.
\end{Remark}

\begin{PROPOSITION}\label{c01}
Let $f: M^m\to\mathbb{ R}^{m+p}, (m\geq 3)$ be a M\"{o}bius homogeneous Wintgen ideal submanifold, then the M\"{o}bius form vanishes, i.e.,$\Phi=0$.
\end{PROPOSITION}
\begin{proof}

If $p=2$, then equation (\ref{nor12}) implies
$4\mu^2=R_{1212}.$ This is a contradiction since $R_{1212}=-(\frac{C^1_1}{\mu}-\frac{C^1_{2,1}}{C^1_1})^2.$
Thus $\Phi=0$ and we finish the proof.

If $p\geq3$. Noticing that $R^{\bot}_{\alpha\beta kl}=0$ for $\alpha=2k+3, \beta=2k+4$, from Lemma\ref{le4}, we have
\begin{equation*}
d\theta_{(2k+3)(2k+4)}=\sum_{\tau}\theta_{(2k+3)\tau}\wedge\theta_{\tau(2k+4)}=2a_k^2\omega_1\wedge\omega_2,
\end{equation*}
and
\begin{equation*}
d\theta_{(2k+3)(2k+4)}=
\left[(k+3)\frac{C^1_{2,1}}{C^1_1}-(k+2)\frac{C^1_1}{\mu}\right]
\left[\frac{C^1_1}{\mu}-\frac{C^1_{2,1}}{C^1_1}\right]
\omega_1\wedge\omega_2.
\end{equation*}
Thus we have
\begin{equation}\label{nork1k}
\left[(k+3)\frac{C^1_{2,1}}{C^1_1}-(k+2)\frac{C^1_1}{\mu}\right]
\left[\frac{C^1_1}{\mu}-\frac{C^1_{2,1}}{C^1_1}\right]=2a_k^2.
\end{equation}
From (\ref{abc2}) we have
$$
\left[(k+3)\frac{C^1_{2,1}}{C^1_1}-(k+2)\frac{C^1_1}{\mu}\right]
\left[\frac{C^1_1}{\mu}-\frac{C^1_{2,1}}{C^1_1}\right]
=(k+2)A_{11}-2(k+3)\mu^2-(k+2)\frac{C^1_1}{2\mu^2}.$$
Since $R_{1212}<0,$ i.e., $A_{11}<2\mu^2+\frac{C^1_1}{2\mu^2}$, Thus
$$(k+2)A_{11}-2(k+3)\mu^2-(k+2)\frac{C^1_1}{2\mu^2}<-2\mu^2.$$
The (\ref{nork1k}) implies
Thus $2a_k^2<-2\mu^2$, which is a  contradiction. Thus $\Phi=0.$
\end{proof}

Since the Wintgen ideal submanifold $f$ is M\"{o}bius homogeneous, then under the orthonormal basis $\{E_1,\cdots,E_m\}$,
\begin{equation}\label{conne}
\begin{split}
&2\omega_{12}+\theta_{12}=-\frac{B^1_{11,3}}{\mu}\omega_3,\\
&\omega_{13}=-\frac{B^1_{11,3}}{\mu}\omega_2,~~~~\omega_{1i}=0, i\geq 4,\\
&\omega_{23}=\frac{B^1_{11,3}}{\mu}\omega_1, ~~~~\omega_{2i}=0, i\geq 4.
\end{split}
\end{equation}
From (\ref{mcof1}), we can obtain the following result,
\begin{PROPOSITION}\label{blasch}
Let $f:M^m\longrightarrow \mathbb{R}^{m+p}(m\geq3)$ be a M\"{o}bius homogeneous Wintgen ideal submanifold. Then we can choose the orthonormal
basis $\{E_1,\cdots,E_m\}$ such that
$$(A_{ij})=diag(A_1,A_1,-A_1,\cdots,-A_1),~~~if~~B^1_{11,3}=0,$$
$$(A_{ij})=diag(A_1,A_1,A_1,-A_1,\cdots,-A_1),~~~if~~B^1_{11,3}\neq0.$$
Particularly, if $B^1_{11,3}\neq0$, $\omega_{3i}=0, i\geq 4$ and
$A_1=\frac{(B^1_{11,3})^2}{2\mu^2}.$
\end{PROPOSITION}

\section{A canonical form of the normal connections}

Before going into the detail, we observe that among the basic examples, the Veronese surfaces in either $S^{2k}$ or $\mathbb{C}P^k$ have a well-known property as being totally isotropic. In particular, the normal bundle of any of these examples has a decomposition into a series of $2$-planes, and the complex normal bundle has a corresponding decomposition into isotropic complex lines. This beautiful structure is also shared by the Wintgen ideal submanifolds constructed from them by generating the cones.

In this section we will show that in the most important cases, a M\"obius homogeneous Wintgen ideal submanifold must have a similar decomposition of the normal bundle. The consequence is that they normal connections can take a canonical form with respect to a good normal frames as demonstrated by two propositions below.

\begin{PROPOSITION}\label{integ0}
Let $f:M^m\longrightarrow \mathbb{R}^{m+p}(m\geq3)$ be a M\"{o}bius homogeneous Wintgen ideal submanifold.
If $B^1_{11,3}=0$ and $\omega_{12}\neq0$,  then there exist basis $\{\xi_3,\xi_4,\cdots,\xi_p\}$ in $Span\{\xi_3,\xi_4,\cdots,\xi_p\}$ such that
the normal connection have the following form
\[
\left\{
\begin{array}{ll}
\left\{
\begin{array}{cc}
\theta_{13}=a_0\omega_2,\\
\theta_{23}=-a_0\omega_1,\\
\end{array}
\right.
\left\{
\begin{array}{cc}
\theta_{14}=a_0\omega_1,\\
\theta_{24}=a_0\omega_2,\\
\end{array}
\right.
\left\{
\begin{array}{cc}
\theta_{1\alpha}=0,~~\alpha\geq 5,\\
\theta_{2\alpha}=0,~~\alpha\geq 5,\\
\end{array}
\right.\\
\theta_{12}-\omega_{12}=\theta_{34}, ~~~a_0\neq0;
\end{array}
\right.
\]
\[
\left\{
\begin{array}{ll}
\left\{
\begin{array}{cc}
\theta_{(2k+1)(2k+3)}=a_k\omega_2,\\
\theta_{(2k+2)(2k+3)}=-a_k\omega_1,\\
\end{array}
\right.
\left\{
\begin{array}{cc}
\theta_{(2k+1)(2k+4)}=a_k\omega_1,\\
\theta_{(2k+2)(2k+4)}=a_k\omega_2,\\
\end{array}
\right.
\left\{
\begin{array}{cc}
\theta_{(2k+1)\alpha}=0,~~\alpha\geq 2k+5,\\
\theta_{(2k+2)\alpha}=0,~~\alpha\geq 2k+5,\\
\end{array}
\right.\\
\theta_{(2k+1)(2k+2)}-\omega_{12}=\theta_{(2k+3)(2k+4)},~~~a_k\neq0;
\end{array}
\right.
\]
$
\left\{
\begin{array}{cc}
\theta_{(2k+3)\alpha}=0,~~\alpha\geq 2k+5,\\
\theta_{(2k+4)\alpha}=0,~~\alpha\geq 2k+5,\\
\end{array}
\right.
$
\end{PROPOSITION}
\begin{proof}
Since $B^1_{11,3}=0$, from (\ref{conne}), we have
\begin{equation}\label{ww}
2\omega_{12}+\theta_{12}=0,~~~d\omega_1=\omega_{12}\wedge\omega_2,~~~d\omega_2=-\omega_{12}\wedge\omega_1.
\end{equation}
We assume  $p\geq 3$. We can rechoose an orthonormal basis in $Span\{\xi_3,\cdots,\xi_p\}$ such that
\begin{equation}\label{2nor131}
\left\{
\begin{array}{cc}
\theta_{13}=\frac{B^3_{11,2}}{\mu}\omega_1+\frac{B^3_{22,1}}{\mu}\omega_2,\\
\theta_{23}=-\frac{B^3_{22,1}}{\mu}\omega_1+\frac{B^3_{11,2}}{\mu}\omega_2,\\
\end{array}
\right.
\left\{
\begin{array}{cc}
\theta_{14}=\frac{B^4_{11,2}}{\mu}\omega_1,\\
\theta_{24}=\frac{B^4_{11,2}}{\mu}\omega_2,\\
\end{array}
\right.
\left\{
\begin{array}{cc}
\theta_{1\alpha}=0,~~\alpha\geq 5,\\
\theta_{2\alpha}=0,~~\alpha\geq 5,\\
\end{array}
\right.
\end{equation}
Using $d\theta_{13}-\sum_{\tau}\theta_{1\tau}\wedge\theta_{\tau3}=
-\frac{1}{2}\sum_{kl}R^{\bot}_{13 kl}\omega_k\wedge\omega_l$ and (\ref{2nor131}), and noting $R^{\bot}_{13 kl}=0$, we obtain
\begin{equation}\label{hi2}
\frac{B^3_{11,2}}{\mu}\Theta\wedge\omega_2-\frac{B^3_{22,1}}{\mu}\Theta\wedge\omega_1+\frac{B^4_{11,2}}{\mu}\theta_{34}\wedge\omega_1=0, ~~~\Theta
=\theta_{12}-\omega_{12}.
\end{equation}
Similarly, we have
\begin{equation}\label{2hi21}
\begin{split}
&\frac{B^3_{22,1}}{\mu}\Theta\wedge\omega_2+\frac{B^3_{11,2}}{\mu}
\Theta\wedge\omega_1-\frac{B^4_{11,2}}{\mu}\theta_{34}\wedge
\omega_2=0,\\
&\frac{B^4_{11,2}}{\mu}\Theta\wedge\omega_1
-\left(-\frac{B^3_{22,1}}{\mu}\omega_1+\frac{B^3_{11,2}}{\mu}\omega_2\right)
\wedge\theta_{34}=0,\\
&\frac{B^4_{11,2}}{\mu}\Theta\wedge\omega_2
+\left(\frac{B^3_{11,2}}{\mu}\omega_1+\frac{B^3_{22,1}}{\mu}\omega_2\right)\wedge\theta_{34}=0.
\end{split}
\end{equation}
From (\ref{hi2}) and (\ref{2hi21}), we eliminate the terms $\omega_1\wedge\theta_{34}$ and $\omega_2\wedge\theta_{34}$, and we get
\begin{equation}\label{2hi22}
\begin{split}
2\frac{B^3_{22,1}B^3_{11,2}}{\mu^2}\Theta\wedge\omega_2+
\left[(\frac{B^3_{11,2}}{\mu})^2-(\frac{B^3_{22,1}}{\mu})^2
+(\frac{B^4_{11,2}}{\mu})^2\right]\Theta\wedge\omega_1=0,\\
\left[(\frac{B^3_{11,2}}{\mu})^2-(\frac{B^3_{22,1}}{\mu})^2
+(\frac{B^4_{11,2}}{\mu})^2\right]\Theta\wedge\omega_2
-2\frac{B^3_{22,1}B^3_{11,2}}{\mu^2}\Theta\wedge\omega_1=0.
\end{split}
\end{equation}
Let $D=[2\frac{B^3_{22,1}B^3_{11,2}}{\mu^2}]^2+[(\frac{B^3_{11,2}}{\mu})^2-(\frac{B^3_{22,1}}{\mu})^2+(\frac{B^4_{11,2}}{\mu})^2]^2$.

If $D\neq0$, then, from (\ref{2hi22}),
we have
$$\Theta=\theta_{12}-\omega_{12}=0.$$
Combining (\ref{ww}), we have
$\omega_{12}=0$, which is in contradiction with the assumption $\omega_{12}\neq0$.
Thus $D=0.$

If $B^3_{22,1}=0$, noting $D=0$, we get $B^3_{11,2}=B^4_{11,2}=0$.
\[
\left\{
\begin{array}{cc}
\theta_{1\alpha}=0,~~\alpha\geq 3,\\
\theta_{2\alpha}=0,~~\alpha\geq 3,\\
\end{array}
\right.
\]
Thus we finish the proof.

If $B^3_{22,1}\neq0$, we get $B^3_{11,2}=0$ and $(B^3_{22,1})^2=(B^4_{11,2})^2$. Let $B^3_{22,1}=B^4_{11,2}\triangleq\mu a_0$, otherwise, take $\tilde{\xi}_4=-\xi_4$.
From (\ref{2nor131}), we have
\[
\left\{
\begin{array}{cc}
\theta_{13}=a_0\omega_2,\\
\theta_{23}=-a_0\omega_1,\\
\end{array}
\right.
\left\{
\begin{array}{cc}
\theta_{14}=a_0\omega_1,\\
\theta_{24}=a_0\omega_2,\\
\end{array}
\right.
\left\{
\begin{array}{cc}
\theta_{1\alpha}=0,~~\alpha\geq 5,\\
\theta_{2\alpha}=0,~~\alpha\geq 5,\\
\end{array}
\right.\]
\[
\theta_{12}-\omega_{12}=\theta_{34}.
\]
Combining the above formula with $d\theta_{1\alpha}=\sum_{\tau}\theta_{1\tau}\wedge
\theta_{\tau\alpha}$, we have
\begin{equation*}
\omega_2\wedge\theta_{3\alpha}+\omega_1\wedge\omega_{4\alpha}=0,~~\alpha\geq 5,~~~-\omega_1\wedge\theta_{3\alpha}+\omega_2\wedge\omega_{4\alpha}=0, ~~\alpha\geq 5.
\end{equation*}
Thus we can assume that
\begin{equation}\label{w12nor}
\left\{
\begin{array}{cc}
\theta_{3\alpha}=a^{\alpha}_1\omega_1+a^{\alpha}_2\omega_2,\\
\theta_{4\alpha}=-a^{\alpha}_2\omega_1+a^{\alpha}_1\omega_2,\\
\end{array}
\right.
\alpha\geq5.
\end{equation}
We can make a new choice of orthonormal frames in $Span\{\xi_5,\cdots,\xi_p\}$ such that
\begin{equation}\label{w12nor1}
\left\{
\begin{array}{cc}
\theta_{35}=a\omega_1+b\omega_2,\\
\theta_{45}=-b\omega_1+a\omega_2,\\
\end{array}
\right.
\left\{
\begin{array}{cc}
\theta_{36}=c\omega_1,\\
\theta_{46}=c\omega_2,\\
\end{array}
\right.
\left\{
\begin{array}{cc}
\theta_{3\alpha}=0,~~\alpha\geq 7,\\
\theta_{4\alpha}=0,~~\alpha\geq 7.\\
\end{array}
\right.
\end{equation}
Using $d\theta_{\alpha\beta}=\sum_{\gamma}\theta_{\alpha\gamma}\wedge\theta_{\gamma\beta}$ for $\alpha=3,4, \beta=5,6$ and (\ref{w12nor1}), we can obtain
that
\begin{equation}\label{w12nor2}
\begin{split}
&a\Theta\wedge\omega_2-b\Theta\wedge\omega_1-c\omega_1\wedge\theta_{56}=0,\\
&b\Theta\wedge\omega_2+a\Theta\wedge\omega_1+c\omega_2\wedge\theta_{56}=0,\\
&c\Theta\wedge\omega_2+[a\omega_1+b\omega_2]\wedge\theta_{56}=0,\\
&c\Theta\wedge\omega_1-[-b\omega_1+a\omega_2]\wedge\theta_{56}=0,
\end{split}
\end{equation}
where $\Theta=\theta_{34}-\omega_{12}$.

From (\ref{w12nor2}), we have
\begin{equation}\label{w12nor3}
\left\{
\begin{array}{cc}
2ab\Theta\wedge\omega_1+(b^2-a^2-c^2)\Theta\wedge\omega_2=0,\\
(a^2-b^2+c^2)\Theta\wedge\omega_1+2ab\Theta\wedge\omega_2=0.\\
\end{array}
\right.
\end{equation}
If $D=4a^2b^2+[a^2-b^2+c^2]^2\neq0$, then (\ref{w12nor3}) implies that $$\Theta=\theta_{34}-\omega_{12}=0.$$ From (\ref{ww})
 $2\omega_{12}+\theta_{12}=0$ and $\theta_{34}=\theta_{12}-\omega_{12}$, we obtain that $\omega_{12}=0$, which is in contradiction with the assumption
$\omega_{12}\neq0.$ Therefore $$D=4a^2b^2+[a^2-b^2+c^2]^2=0.$$

If $b=0$, combining with $D=4a^2b^2+[a^2-b^2+c^2]^2=0$, we have $a=c=0$. Thus we finish the proof.

If $b\neq0$, then combining with $D=4a^2b^2+[a^2-b^2+c^2]^2=0$, we have $a=0$ and $b=\pm c$. Let $b=c\triangleq a_1$, otherwise $\tilde{\xi}_6=-\xi_6$. From (\ref{w12nor1}),
we have
\begin{equation}\label{w12nor4}
\left\{
\begin{array}{cc}
\theta_{35}=a_1\omega_2,\\
\theta_{45}=-a_1\omega_1,\\
\end{array}
\right.
\left\{
\begin{array}{cc}
\theta_{36}=a_1\omega_1,\\
\theta_{46}=a_1\omega_2,\\
\end{array}
\right.
\left\{
\begin{array}{cc}
\theta_{3\alpha}=0,~~\alpha\geq 7,\\
\theta_{4\alpha}=0,~~\alpha\geq 7,\\
\end{array}
\right.
\theta_{56}=\theta_{34}-\omega_{12}.
\end{equation}
Repeating the process (\ref{w12nor}--\ref{w12nor4}), we finish the proof.
\end{proof}
\begin{PROPOSITION}\label{3d}
Let $f:M^3\longrightarrow \mathbb{R}^{3+p}$ be a M\"{o}bius homogeneous Wintgen ideal submanifold.
If $B^1_{11,3}\neq0$,  then there exist basis $\{\xi_3,\xi_4,\cdots,\xi_p\}$ in $Span\{\xi_3,\xi_4,\cdots,\xi_p\}$ such that
the normal connection have the following form
\[
\left\{
\begin{array}{ll}
\left\{
\begin{array}{cc}
\theta_{13}=a_0\omega_2,\\
\theta_{23}=-a_0\omega_1,\\
\end{array}
\right.
\left\{
\begin{array}{cc}
\theta_{14}=a_0\omega_1,\\
\theta_{24}=a_0\omega_2,\\
\end{array}
\right.
\left\{
\begin{array}{cc}
\theta_{1\alpha}=0,~~\alpha\geq 5,\\
\theta_{2\alpha}=0,~~\alpha\geq 5,\\
\end{array}
\right.\\
\theta_{12}-\omega_{12}-\frac{B^1_{11,3}}{\mu}\omega_3=\theta_{34},~~~a_0\neq0;
\end{array}
\right.
\]
\[
\left\{
\begin{array}{ll}
\left\{
\begin{array}{cc}
\theta_{(2k+1)(2k+3)}=a_k\omega_2,\\
\theta_{(2k+2)(2k+3)}=-a_k\omega_1,\\
\end{array}
\right.
\left\{
\begin{array}{cc}
\theta_{(2k+1)(2k+4)}=a_k\omega_1,\\
\theta_{(2k+2)(2k+4)}=a_k\omega_2,\\
\end{array}
\right.
\left\{
\begin{array}{cc}
\theta_{(2k+1)\alpha}=0,~~\alpha\geq 2k+5,\\
\theta_{(2k+2)\alpha}=0,~~\alpha\geq 2k+5,\\
\end{array}
\right.\\
\theta_{(2k+1)(2k+2)}-\omega_{12}-\frac{B^1_{11,3}}{\mu}\omega_3=\theta_{(2k+3)(2k+4)},~~~ a_k\neq0,
\end{array}
\right.
\]
$\left\{
\begin{array}{cc}
\theta_{(2k+3)\alpha}=0,~~\alpha\geq 2k+5,\\
\theta_{(2k+4)\alpha}=0,~~\alpha\geq 2k+5.\\
\end{array}
\right.$
\end{PROPOSITION}
\begin{proof} We use induction for $k$ to prove the Proposition~ \ref{3d}.
Choose a new orthonormal basis in $Span\{\xi_3,\cdots,\xi_p\}$ such that
\begin{equation}\label{nor31}
\left\{
\begin{array}{cc}
\theta_{13}=\frac{B^3_{11,2}}{\mu}\omega_1+\frac{B^3_{22,1}}{\mu}\omega_2,\\
\theta_{23}=-\frac{B^3_{22,1}}{\mu}\omega_1+\frac{B^3_{11,2}}{\mu}\omega_2,\\
\end{array}
\right.
\left\{
\begin{array}{cc}
\theta_{14}=\frac{B^4_{11,2}}{\mu}\omega_1,\\
\theta_{24}=\frac{B^4_{11,2}}{\mu}\omega_2,\\
\end{array}
\right.
\left\{
\begin{array}{cc}
\theta_{1\alpha}=0,~~\alpha\geq 5,\\
\theta_{2\alpha}=0,~~\alpha\geq 5,\\
\end{array}
\right.
\end{equation}
Using $d\theta_{\alpha\beta}-\sum_{\tau}\theta_{\alpha\tau}\wedge\theta_{\tau\beta}=
-\frac{1}{2}\sum_{kl}R^{\bot}_{\alpha\beta kl}\omega_k\wedge\omega_l$ and (\ref{nor31}), and noting $R^{\bot}_{\alpha\beta kl}=0$ for $\alpha=1,2,\beta=3,4$, we obtain
\begin{equation}\label{no31}
\begin{split}
&\frac{B^3_{11,2}}{\mu}\Theta\wedge\omega_2-\frac{B^3_{22,1}}{\mu}\Theta\wedge\omega_1+\frac{B^4_{11,2}}{\mu}\theta_{34}\wedge\omega_1=0,\\
&\frac{B^3_{22,1}}{\mu}\Theta\wedge\omega_2+\frac{B^3_{11,2}}{\mu}\Theta\wedge\omega_1-\frac{B^4_{11,2}}{\mu}\theta_{34}\wedge\omega_2=0,\\
&\frac{B^4_{11,2}}{\mu}\Theta\wedge\omega_1-
\left(-\frac{B^3_{22,1}}{\mu}\omega_1+\frac{B^3_{11,2}}{\mu}\omega_2\right)
\wedge\theta_{34}=0,\\
&\frac{B^4_{11,2}}{\mu}\Theta\wedge\omega_2+
\left(\frac{B^3_{11,2}}{\mu}\omega_1+\frac{B^3_{22,1}}{\mu}\omega_2\right)
\wedge\theta_{34}=0,
\end{split}
\end{equation}
where $\Theta
=\theta_{12}-\omega_{12}-\frac{B^1_{11,3}}{\mu}\omega_3,$

From (\ref{no31}), we eliminate the terms $\omega_1\wedge\theta_{34}$ and $\omega_2\wedge\theta_{34}$, and we get
\begin{equation}\label{no321}
\begin{split}
2\frac{B^3_{22,1}B^3_{11,2}}{\mu^2}\Theta\wedge\omega_2+
\left[\left(\frac{B^3_{11,2}}{\mu}\right)^2
-\left(\frac{B^3_{22,1}}{\mu}\right)^2
+\left(\frac{B^4_{11,2}}{\mu}\right)^2\right]\Theta\wedge\omega_1=0,\\
\left[\left(\frac{B^3_{11,2}}{\mu}\right)^2
-\left(\frac{B^3_{22,1}}{\mu}\right)^2
+\left(\frac{B^4_{11,2}}{\mu}\right)^2\right]\Theta\wedge\omega_2-2\frac{B^3_{22,1}B^3_{11,2}}{\mu^2}\Theta\wedge\omega_1=0.
\end{split}
\end{equation}
{\bf Claim1}: $D=[2\frac{B^3_{22,1}B^3_{11,2}}{\mu^2}]^2+[(\frac{B^3_{11,2}}{\mu})^2-(\frac{B^3_{22,1}}{\mu})^2+(\frac{B^4_{11,2}}{\mu})^2]^2$=0.

Proof of Claim1: Since $|B|^2=\frac{m-1}{m}$, using covariant derivative of
$B^{\alpha}_{ij,k}$ and $\Phi=0$, we have
\begin{equation}\label{lb}
\begin{split}
0=\frac{1}{2}\Delta|B|^2&=\sum_{\alpha,i,j,k}|B^{\alpha}_{ij,k}|^2+\sum_{\alpha,i,j,k,l}B^{\alpha}_{ij}B^{\alpha}_{kl}R_{kijl}\\
&+\sum_{\alpha,i,j,k,l}B^{\alpha}_{ij}B^{\alpha}_{ik}R_{jlkl}-\sum_{\alpha,\beta,i,j,k}B^{\alpha}_{ij}B^{\beta}_{ki}R^{\bot}_{\alpha\beta jk}.
\end{split}
\end{equation}
If $D\neq0$, from (\ref{no321}) we have
$\Theta=0,$ $\theta_{12}-\omega_{12}-\frac{B^1_{11,3}}{\mu}\omega_3=0.$ Combining $2\omega_{12}+\theta_{12}+\frac{B^1_{11,3}}{\mu}\omega_3=0$, we have
\begin{equation}\label{wn1}
\omega_{12}=-\frac{2B^1_{11,3}}{3\mu}\omega_3,~~~~~~~\theta_{12}=\frac{B^1_{11,3}}{3\mu}\omega_3.
\end{equation}
Using $d\omega_3=-2\frac{B^1_{11,3}}{\mu}\omega_1\wedge\omega_2$, $d\omega_{12}-\sum_k\omega_{1k}\wedge\omega_{k2}=-\frac{1}{2}\sum_{kl}R_{12kl}\omega_k\wedge\omega_l$ and
$d\theta_{12}-\sum_k\theta_{1k}\wedge\theta_{k2}=-\frac{1}{2}\sum_{kl}R^{\bot}_{12kl}\omega_k\wedge\omega_l$, we obtain
\begin{equation}\label{curn1}
R_{1212}=-\frac{7}{3}\left(\frac{B^1_{11,3}}{\mu}\right)^2,~~~
R^{\bot}_{1212}=\frac{2}{3}\left(\frac{B^1_{11,3}}{\mu}\right)^2
-\left(\frac{B^3_{22,1}}{\mu}\right)^2
+\left(\frac{B^3_{11,2}}{\mu}\right)^2
+\left(\frac{B^4_{11,2}}{\mu}\right)^2.
\end{equation}
Combining (\ref{equa4}), we have
\begin{equation}\label{curn2}
\mu^2=\frac{5}{3}\left(\frac{B^1_{11,3}}{\mu}\right)^2,~~~
4\left(\frac{B^1_{11,3}}{\mu}\right)^2=
\left(\frac{B^3_{22,1}}{\mu}\right)^2+
\left(\frac{B^3_{11,2}}{\mu}\right)^2+
\left(\frac{B^4_{11,2}}{\mu}\right)^2.
\end{equation}
Combining (\ref{lb}), (\ref{curn1}) and (\ref{curn2}), we get $\mu=0$, which is a contradiction. so $D=0$.

If $B^3_{22,1}=0$. Since $D=0$, then $B^3_{11,2}=B^4_{11,2}=0$, that is
$$\theta_{1\alpha}=0,~~\alpha\geq3,~~~~~\theta_{2\alpha}=0,~~\alpha\geq3.$$
This implies that the codimension of $f$ can reduce to $2$ and $a_k=0$ for $k=0$.

If $B^3_{22,1}\neq 0$. Since $D=0$, then $B^3_{11,2}=0$ and $(B^3_{22,1})^2=(B^4_{11,2})^2$. We can assume that $B^3_{22,1}=B^4_{11,2}$, otherwise, take
$\tilde{\xi}_4=-\xi_4$. Let $a_0=\frac{B^4_{11,2}}{\mu}$. Thus
\[
\left\{
\begin{array}{cc}
\theta_{13}=a_0\omega_2,\\
\theta_{23}=-a_0\omega_1,\\
\end{array}
\right.
\left\{
\begin{array}{cc}
\theta_{14}=a_0\omega_1,\\
\theta_{24}=a_0\omega_2,\\
\end{array}
\right.
\left\{
\begin{array}{cc}
\theta_{1\alpha}=0,~~\alpha\geq 5,\\
\theta_{2\alpha}=0,~~\alpha\geq 5,\\
\end{array}
\right.
\]
From (\ref{no31}), we get
$$\theta_{34}=\Theta=\theta_{12}-\omega_{12}-\frac{B^1_{11,3}}{\mu}\omega_3.$$

Next we assume
\begin{equation}\label{rema}
\left\{
\begin{array}{cc}
\theta_{(2k-1)(2k+1)}=a_{k-1}\omega_2,\\
\theta_{(2k)(2k+1)}=-a_{k-1}\omega_1,\\
\end{array}
\right.
\left\{
\begin{array}{cc}
\theta_{(2k-1)(2k+2)}=a_{k-1}\omega_1,\\
\theta_{(2k)(2k+2)}=a_{k-1}\omega_2,\\
\end{array}
\right.
\end{equation}
\begin{equation}\label{rema1}
\left\{
\begin{array}{cc}
\theta_{(2k-1)\alpha}=0,~~\alpha\geq 2k+3,\\
\theta_{(2k)\alpha}=0,~~\alpha\geq 2k+3,\\
\end{array}
\right.
\end{equation}
\begin{equation}\label{rema2}
\theta_{(2k-1)(2k)}-\omega_{12}-\frac{B^1_{11,3}}{\mu}\omega_3=\theta_{(2k+1)(2k+2)}.
\end{equation}
Using $d\theta_{(2k-1)\alpha}-\sum_{\tau}\theta_{(2k-1)\tau}\wedge\theta_{\tau\alpha}=
-\frac{1}{2}\sum_{sl}R^{\bot}_{(2k-1)\alpha sl}\omega_s\wedge\omega_l$ and (\ref{rema1}), and noting $R^{\bot}_{(2k-1)\alpha sl}=0$, we obtain
\begin{equation}\label{remac1}
\left\{
\begin{array}{cc}
\omega_2\wedge\theta_{(2k+1)\alpha}+\omega_1\wedge\theta_{(2k+2)\alpha}=0,~~\alpha\geq 2k+3,\\
-\omega_1\wedge\theta_{(2k+1)\alpha}+\omega_2\wedge\theta_{(2k+2)\alpha}=0,~~\alpha\geq 2k+3.
\end{array}
\right.
\end{equation}
From (\ref{remac1}), we can assume that
\[
\left\{
\begin{array}{cc}
\theta_{(2k+1)\alpha}=a^{\alpha}_1\omega_1+a^{\alpha}_2\omega_2,~~\alpha\geq 2k+3,\\
\theta_{(2k+2)\alpha}=-a^{\alpha}_2\omega_1+a^{\alpha}_1\omega_2,~~\alpha\geq 2k+3.
\end{array}
\right.
\]
Furthermore, we can choose a basis $\{\xi_{2k+3},\cdots,\xi_p\}$ in $Span\{\xi_{2k+3},\cdots,\xi_p\}$ such that
\begin{equation}\label{norc1}
\left\{
\begin{array}{cc}
\theta_{(2k+1)(2k+3)}=a\omega_1+b\omega_2,\\
\theta_{(2k+2)(2k+3)}=-b\omega_1+a\omega_2,
\end{array}
\right.
\end{equation}
\begin{equation}\label{norc2}
\left\{
\begin{array}{cc}
\theta_{(2k+1)(2k+4)}=c\omega_1,\\
\theta_{(2k+2)(2k+4)}=c\omega_2,\\
\end{array}
\right.
\left\{
\begin{array}{cc}
\theta_{(2k+1)\alpha}=0,~~\alpha\geq 2k+5,\\
\theta_{(2k+2)\alpha}=0,~~\alpha\geq 2k+5.
\end{array}
\right.
\end{equation}
Using $d\theta_{(2k+1)(2k+3)}-\sum_{\tau}\theta_{(2k+1)\tau}\wedge\theta_{\tau(2k+3)}=
-\frac{1}{2}\sum_{sl}R^{\bot}_{(2k+1)(2k+3) sl}\omega_s\wedge\omega_l$ and (\ref{norc1}), and noting $R^{\bot}_{(2k+1)(2k+3) sl}=0$, we obtain
\begin{equation}\label{noy}
a\Theta\wedge\omega_2-b\Theta\wedge\omega_1+c\theta_{(2k+3)(2k+4)}\wedge\omega_1=0, ~~\Theta
=\theta_{(2k+1)(2k+2)}-\omega_{12}-\frac{B^1_{11,3}}{\mu}\omega_3.
\end{equation}
Similarly, we have
\begin{equation}\label{noy1}
\begin{split}
&b\Theta\wedge\omega_2+a\Theta\wedge\omega_1-c\theta_{(2k+3)(2k+4)}\wedge\omega_2=0,\\
&c\Theta\wedge\omega_1-(-b\omega_1+a\omega_2)\wedge\theta_{(2k+3)(2k+4)}=0,\\
&c\Theta\wedge\omega_2+(a\omega_1+b\omega_2)\wedge\theta_{(2k+3)(2k+4)}=0.
\end{split}
\end{equation}
From (\ref{noy}) and (\ref{noy1}), we eliminate the terms $\omega_1\wedge\theta_{(2k+3)(2k+4)}$ and $\omega_2\wedge\theta_{(2k+3)(2k+4)}$, and we get
\begin{equation}\label{no321}
\begin{split}
2ab\Theta\wedge\omega_2+[a^2-b^2+c^2]\Theta\wedge\omega_1=0,\\
[a^2-b^2+c^2]\Theta\wedge\omega_2-2ab\Theta\wedge\omega_1=0.
\end{split}
\end{equation}
Like Claim~1, we can prove that $$D=[2ab]^2+[a^2-b^2+c^2]^2=0.$$
If $b=0$. Since $D=0$, then $a=c=0$, which implies that the codimension of $f$ can reduce to $2k+2$ and $a_k=0$.

If $b\neq0$. Since $D=0$, then $a=0$ and $b^2=c^2$. We can assume $b=c=a_k$, thus
\[
\left\{
\begin{array}{cc}
\theta_{(2k+1)(2k+3)}=a_k\omega_2,\\
\theta_{(2k+2)(2k+3)}=-a_k\omega_1,\\
\end{array}
\right.
\left\{
\begin{array}{cc}
\theta_{(2k+1)(2k+4)}=a_k\omega_1,\\
\theta_{(2k+2)(2k+4)}=a_k\omega_2,\\
\end{array}
\right.
\left\{
\begin{array}{cc}
\theta_{(2k+1)\alpha}=0,~~\alpha\geq 2k+5,\\
\theta_{(2k+2)\alpha}=0,~~\alpha\geq 2k+5,\\
\end{array}
\right.
\]
Since $a=0$ and $b=c$, From (\ref{noy}) we have
\[
\theta_{(2k+3)(2k+4)}=\Theta=\theta_{(2k+1)(2k+2)}-\omega_{12}-\frac{B^1_{11,3}}{\mu}\omega_3.
\]
Thus we finish the proof to Proposition~\ref{3d}.
\end{proof}

\section{The reduction to $2$ and $3$ dimensional case}

To obtain the classification of M\"obius homogeneous Wintgen ideal submanifolds, a crucial observation is that they can be reduced to $2$ or $3$ dimensional case via the construction of cones in Section~4.
The reader may compare this to our general results for Wintgen ideal submanifolds with a low-dimensional integrable distribution \cite{LiTZ2}, and the reduction theorem for hypersurfaces in M\"obius geometry \cite{LiTZ1}.

\begin{lemma}\label{integ1}
Let $f:M^m\longrightarrow \mathbb{R}^{m+p} (m\geq3)$ be a M\"{o}bius homogeneous Wintgen ideal submanifold.
Then $A_1>0.$
\end{lemma}
\begin{proof}
If $B^1_{11,3}\neq0$, from Proposition~\ref{blasch}, $A_1=\frac{(B^1_{11,3})^2}{2\mu^2}>0$. Thus we need only to consider the case when $B^1_{11,3}=0$.

If $B^1_{11,3}=0,\omega_{12}=0,$ then $R_{1212}=0$, i.e.,$-2\mu^2+2A_1=0$.
Thus $A_1=\mu^2>0.$

Consider the case $B^1_{11,3}=0, \omega_{12}\neq0$. From Proposition~\ref{integ0}, we have
\begin{equation}\label{nork}
\theta_{12}-(k+1)\omega_{12}=\theta_{(2k+3)(2k+4)}.
\end{equation}
 From (\ref{conne}), we have
\begin{equation}\label{conne1}
2\omega_{12}+\theta_{12}=0,~~~~~
\omega_{1i}=0, i\geq 3,~~~~~~
\omega_{2i}=0, i\geq 3.
\end{equation}
Thus, we have
\begin{equation}\label{nork1}
\theta_{(2k+3)(2k+4)}=-(k+3)\omega_{12}.
\end{equation}
Using Proposition (\ref{integ0}) and (\ref{nork1}), we have
\begin{equation*}
d\theta_{(2k+3)(2k+4)}=\sum_{\tau}\theta_{(2k+3)\tau}\wedge\theta_{\tau(2k+4)}=2a_k^2\omega_1\wedge\omega_2
=(k+3)R_{1212}\omega_1\wedge\omega_2,
\end{equation*}
which implies
\begin{equation}\label{nork34}
2a_k^2=(k+2)R_{1212}=(k+3)(-2\mu^2+2A_1).
\end{equation}
 Thus $A_1>0.$
\end{proof}
\begin{PROPOSITION}\label{integ11}
Let $f:M^m\longrightarrow \mathbb{R}^{m+p} (m\geq3)$ be a M\"{o}bius homogeneous Wintgen ideal submanifold.
If $B^1_{11,3}=0$, then locally $f$ is M\"{o}bius equivalent to a cone over a surface $u:M^2\to \mathbb{S}^{2+p}$.
\end{PROPOSITION}
\begin{proof} From (\ref{conne}), we have
\begin{equation}\label{conne1}
2\omega_{12}+\theta_{12}=0,~~~~~
\omega_{1i}=0, i\geq 3,~~~~~~
\omega_{2i}=0, i\geq 3.
\end{equation}
Since $d\omega_a\equiv0, mod\{\omega_3,\cdots,\omega_m\}, a\geq 3$, $\mathbb{D}=span\{E_1,E_2\}$ is integrable.
Using (\ref{conne1}) and Proposition~\ref{blasch}, we have
\begin{equation}\label{frame1}
\begin{split}
d\xi_1=-\mu\omega_1Y_2-\mu\omega_2Y_1+\sum_{\alpha=1}\theta_{1\alpha}\xi_\alpha,\\
d\xi_2=-\mu\omega_1Y_1+\mu\omega_2Y_2+\sum_{\alpha=1}\theta_{2\alpha}\xi_\alpha,\\
d\xi_{\alpha}=-\theta_{1\alpha}\xi_1-\theta_{2\alpha}\xi_2+\sum_{\beta}\theta_{\alpha\beta}\xi_{\beta}.
\end{split}
\end{equation}
\begin{equation}\label{frame2}
\begin{split}
&dY_1=\omega_{12}Y_2
-\omega_1\left(A_1Y+N\right)
+\mu\omega_2\xi_1+\mu\omega_1\xi_2,\\
&dY_2=-\omega_{12}Y_1
-\omega_2\left(A_1Y+N\right)
+\mu\omega_1\xi_1-\mu\omega_2\xi_2,\\
&d\left(A_1Y+N\right)=2A_1[\omega_1Y_1+\omega_2Y_2]
\end{split}
\end{equation}
From (\ref{frame1}) and (\ref{frame2}), we know that the subspace
$$V=span\{(A_1Y+N),Y_1,Y_2,\xi_1,\xi_2,\cdots,\xi_p\}$$
is parallel along $M^m$. The orthogonal complement $V^{\perp}$ also is parallel along $M^m$. In fact,
$$V^{\perp}=span\{(A_1Y-N),Y_3,\cdots,Y_m\}.$$
 Using (\ref{conne1}), we can obtain
\begin{equation}\label{t1}
d(A_1Y-N)=2A_1\sum_{a\geq3}\omega_aY_a, ~~~dY_a=\omega_a(A_1Y-N)+\sum_{b\geq3}\omega_{ab}Y_b,~~~~a\geq3.
\end{equation}

Since $d\omega_1=\omega_{12}\wedge\omega_1, d\omega_2=-\omega_{12}\wedge_1$, the distribution $\mathbb{D}^{\perp}=span\{E_3,\cdots,E_m\}$ also is integrable. From (\ref{frame1}) and  (\ref{frame2}),
we know that the mean curvature spheres $\xi_1,\xi_2$ induce $2$-dimensional submanifolds in the de sitter space $\mathbb{S}_1^{m+p+1}$
$$\xi_1,\xi_2:M^2=M^m/F\longrightarrow \mathbb{S}_1^{m+p+1},$$
where fibers $F$ are integral submanifolds of distribution $\mathbb{D}^{\perp}$. In other words, $\xi_1,\xi_2$ form $2$-parameter family of $(m+p-1)$-spheres enveloped
by $f:M^m\longrightarrow \mathbb{R}^{m+p}$.

Since $\langle (A_1Y+N),(A_1Y+N)\rangle =2A_1>0$,  $V$ is a fixed space-like subspace, $V^{\bot}$ is a fixed Lorentz subspace in $\mathbb{R}^{m+p+2}_1$.
We can assume that $V=\mathbb{R}^{3+p}, ~~V^{\bot}=\mathbb{R}^{m-1}_1$. From (\ref{frame1}) and (\ref{frame2}), we know
$$u=\frac{1}{\sqrt{A_1}}(A_1Y+N):M^2\to \mathbb{S}^{2+p}.$$
On the other hand, the equation (\ref{t1}) implies that
$$\phi=\frac{1}{\sqrt{A_1}}(A_1Y-N):\mathbb{H}^{m-2}\to \mathbb{R}^{m-1}_1$$
is the standard embedding of the hyperbolic space $\mathbb{H}^{m-2}$ in $\mathbb{R}^{m-1}_1$.
Then
\begin{equation*}
Y=2\sqrt{A_1}(u,\phi):M^2\times \mathbb{H}^{m-2}\to \mathbb{S}^{2+p}\times \mathbb{H}^{m-2}\subset \mathbb{R}^{m+p+2}_1,
\end{equation*}
where $\phi:\mathbb{H}^{m-2}\to \mathbb{H}^{m-2}$ is the identity map.
From Proposition~\ref{redu}, we know that $f$ is a cone over $u:M^2\to \mathbb{S}^{2+p}.$
We complete the proof of Proposition \ref{integ11}.
\end{proof}
\begin{PROPOSITION}\label{integ2}
Let $f:M^m\longrightarrow \mathbb{R}^{m+p} (m\geq4)$ be a M\"{o}bius homogeneous Wintgen ideal submanifold.
If $B^1_{11,3}\neq0$,  then locally $f$ is M\"{o}bius equivalent to
a cone over a three dimensional M\"{o}bius homogeneous Wintgen ideal submanifold in $\mathbb{S}^{3+p}$.
\end{PROPOSITION}
\begin{proof}
From (\ref{conne}) and Proposition~\ref{blasch}, we have
\begin{equation}\label{conne12}
\begin{split}
&2\omega_{12}+\theta_{12}=\frac{-B^1_{11,3}}{\mu}\omega_3,~~
\omega_{13}=\frac{-B^1_{11,3}}{\mu}\omega_2,~~\omega_{23}=\frac{B^1_{11,3}}{\mu}\omega_1,\\
&\omega_{1i}=0, i\geq 4,~~~~~~
\omega_{2i}=0, i\geq 4,~~~~~~\omega_{3i}=0,  i\geq 4.
\end{split}
\end{equation}
Since $d\omega_a\equiv0, mod\{\omega_4,\cdots,\omega_m\}, a\geq 4,$ $\mathbb{D}=span\{E_1,E_2,E_2\}$ is integrable,
Using (\ref{conne12}) and Proposition~\ref{blasch}, we have
\begin{equation}\label{frame12}
\begin{split}
d\xi_1=-\mu\omega_1Y_2-\mu\omega_2Y_1+\sum_{\alpha=1}\theta_{1\alpha}\xi_{\alpha},\\
d\xi_2=-\mu\omega_1Y_1+\mu\omega_2Y_2+\sum_{\alpha=1}\theta_{2\alpha}\xi_{\alpha},\\
d\xi_{\alpha}=-\theta_{1\alpha}\xi_1-\theta_{2\alpha}\xi_2+\sum_{\beta}\theta_{\alpha\beta}\xi_{\beta}.
\end{split}
\end{equation}
\begin{equation}\label{frame22}
\begin{split}
&dY_1=-\omega_1\left(A_1Y+N\right)+\omega_{12}Y_2+\omega_{13}Y_3
+\mu\omega_2\xi_1+\mu\omega_1\xi_2,\\
&dY_2=-\omega_2\left(A_1Y+N\right)-\omega_{12}Y_1+\omega_{23}Y_3
+\mu\omega_1\xi_1-\mu\omega_2\xi_2,\\
&dY_3=-\omega_3\left(A_1Y+N\right)-\omega_{13}Y_1-\omega_{23}Y_2,\\
&d\left(A_1Y+N\right)=2A_1[\omega_1Y_1+\omega_2Y_2+\omega_3Y_3].
\end{split}
\end{equation}
From (\ref{frame12}) and (\ref{frame22}), we know that the subspace
$$V=span\{(A_1Y+N),Y_1,Y_2,Y_3,\xi_1,\xi_2,\cdots,\xi_p\}$$
is parallel along $M^m$. The orthogonal complement $V^{\perp}$ also is parallel along $M^m$. In fact,
$$V^{\perp}=span\{(A_1Y-N),Y_4,\cdots,Y_m\}.$$
 Using (\ref{conne1}), we can obtain
\begin{equation}\label{t}
d(A_1Y-N)=2A_1\sum_{a\geq4}\omega_aY_a,~~~dY_a=\omega_a(A_1Y-N)+\sum_{b\geq 4}\omega_{ab}Y_b, ~~a\geq 4.
\end{equation}
Since $\langle (A_1Y+N),(A_1Y+N)\rangle =2A_1>0$,  $V$ is a fixed space-like subspace.
Like as the proof of Proposition~ \ref{integ11}, we can prove that $f$ is M\"{o}bius equivalent to
a cone over a three dimensional Wintgen ideal submanifold $u:M^3\to \mathbb{S}^{3+p}$. Since $f$ is M\"{o}bius homogeneous,
clearly $u:M^3\to \mathbb{S}^{3+p}$ is also M\"{o}bius homogeneous.
\end{proof}

\section{Proof of the Main theorem}

\begin{PROPOSITION}\label{the2}
Let $f:M^m\longrightarrow \mathbb{R}^{m+p} (m\geq3)$ be a M\"{o}bius homogeneous Wintgen ideal submanifold.
If $B^1_{11,3}=0$,  then locally $f$ is M\"{o}bius equivalent to
\begin{description}
\item (i) a cone over a Veronese surface in ${\mathbb S}^{2k}$,
\item (ii) a cone over a flat minimal surface in ${\mathbb S}^{2k+1}$.
\end{description}
\end{PROPOSITION}
\begin{proof}
From Proposition~\ref{integ11}, we know that $f$ is a cone over $u:M^2\to \mathbb{S}^{2+p}.$ Since the M\"{o}bius form of $f$ vanishes, from Proposition~ \ref{redu1}, we know that
the M\"{o}bius form of the surface $u$ vanishes. The surfaces with vanishing M\"{o}bius form is classified in \cite{li}.
We complete the proof to Proposition~ \ref{the2}.
\end{proof}
If $B^1_{11,3}\neq0$, by Proposition~ \ref{integ2} we need only to consider three dimensional M\"{o}bius homogeneous Wintgen ideal submanifolds in $S^{3+p}$.

\begin{PROPOSITION}\label{integ3}
Let $x:M^3\longrightarrow \mathbb{S}^{3+p}$ be a M\"{o}bius homogeneous Wintgen ideal submanifold.
If $B^1_{11,3}\neq0$, then locally $x$ is M\"{o}bius equivalent to the M\"{o}bius homogeneous Wintgen ideal submanifold given by Example \ref{ex3}.
\end{PROPOSITION}
\begin{proof}
Let $\sigma:\mathbb{S}^{3+p}\to \mathbb{R}^{3+p}$ the stereographic projection. From \cite{liu}, we know that the submanifolds $x:M^3\longrightarrow \mathbb{S}^{3+p}$
and $f=\sigma\circ x:M^3\longrightarrow \mathbb{R}^{3+p}$ have the same M\"{o}bius invariants, especially, the normal
connection.
From Proposition~ \ref{3d}, we can assume $f: ~M^3 \rightarrow \mathbb{S}^{2n+3}$ and there exists basis $\{\xi_1, \xi_2, \cdots, \xi_{2n}\}$ such that the normal
connection under this frame has the following forms
\[
\left\{
\begin{array}{cc}
\theta_{(2k+1)(2k+3)}=a_k\omega_2,\\
\theta_{(2k+2)(2k+3)}=-a_k\omega_1,\\
\end{array}
\right.
\left\{
\begin{array}{cc}
\theta_{(2k+1)(2k+4)}=a_k\omega_1,\\
\theta_{(2k+2)(2k+4)}=a_k\omega_2,\\
\end{array}
\right.
\left\{
\begin{array}{cc}
\theta_{(2k+1)\alpha}=0,~2k+5\leq\alpha\leq 2n,\\
\theta_{(2k+2)\alpha}=0,~2k+5\leq\alpha\leq 2n,\\
\end{array}
\right.
\]
\begin{equation*}\theta_{34}=\theta_{12}-\omega_{12}-\frac{B^1_{11,3}}{\mu}\omega_3,~~~
\theta_{(2k+3)(2k+4)}=\theta_{(2k+1)(2k+2)}-\omega_{12}-\frac{B^1_{11,3}}{\mu}\omega_3,
\end{equation*}
where $k=0,1,2,\cdots, n-2$.

From the preceding discussion, we know that when $B^1_{11,3}\neq0$
$$\omega_{3i}=0, i\geq4;~~~(A_{ij})=diag(A_1,A_1,A_1,-A_1,\cdots,-A_1);$$
where $A_1=\frac{(B^1_{11,3})^2}{2\mu^2}.$ Without lost of generality, we assume $\sqrt{2A_1}=\frac{B^1_{11,3}}{\mu}\doteq L.$
In the following, we define $$\eta=\frac{A_1+N}{\sqrt{2A_1}}.$$ In fact we have $\eta: M^3\rightarrow \mathbb{S}^{2n+3},$ this follows from the following structure equations \eqref{Theta1}.

Combining \eqref{frame12} and \eqref{frame22}, with respect to the frame $\{\eta, Y_3, Y_1, Y_2, \xi_1, \xi_2, \cdots,\xi_{2n-1} \xi_{2n}\}$
we can write out the structure equations as follows:
\begin{small}\begin{equation}\label{Theta1}
d\begin{pmatrix}
\eta\\ Y_3\\ Y_1\\ Y_2\\ \xi_1\\ \xi_2\\ \vdots\\ \xi_{2n-1} \\ \xi_{2n}\end{pmatrix}=\Theta
\begin{pmatrix}
\eta\\ Y_3\\ Y_1\\ Y_2\\ \xi_1\\ \xi_2\\ \vdots\\ \xi_{2n-1} \\ \xi_{2n}\end{pmatrix},
\end{equation}\end{small}
where $\Theta=$
\[
\left(\begin{array}{cccccccccc}
0& L\omega_3& L\omega_1& L\omega_2& 0& 0& 0& 0& \cdots& \vec{0}\\
-L\omega_3& 0& -L\omega_2& L\omega_1& 0& 0& 0& 0& \cdots& \vec{0}\\
-L\omega_1& L\omega_2& 0& \omega_{12}& \mu\omega_2& -\mu\omega_1& 0& 0 &\cdots& \vec{0}\\
-L\omega_2& -L\omega_1& -\omega_{12}& 0& \mu\omega_1& \mu\omega_2& 0& 0 & \cdots& \vec{0}\\
0& 0& -\mu\omega_2& -\mu\omega_1& 0& \theta_{12} &a_0\omega_2 &a_0\omega_1 &\cdots &\vec{0}\\
0& 0&  \mu\omega_1& -\mu\omega_2& -\theta_{12}& 0 &-a_0\omega_1 &a_0\omega_2&\cdots &\vec{0}\\
0&0&0&0&-a_0\omega_2& a_0\omega_1& 0& \theta_{34} &\cdots&\vec{0}\\
0&0&0&0&-a_0\omega_1& -a_0\omega_2& -\theta_{34}&0 &\cdots&\vec{0}\\
\vdots&\vdots&\vdots&\vdots&\vdots&\vdots&\vdots&\vdots&\ddots&\vec{0}\\
\vec{0}& \vec{0}& \vec{0}& \vec{0}& \vec{0}&\vec{ 0}&\vec{0}&\vec{0}&\cdots&B
\end{array}
\right)
\]
where $$\bf{\Huge B}=
\begin{pmatrix}
-a_{n-2}\omega_2& a_{n-2}\omega_1& 0& \theta_{2n-1\,2n}\\
-a_{n-2}\omega_1& -a_{n-2}\omega_2& -\theta_{2n-1\,2n}&0
\end{pmatrix}.
$$
Denote the frame as a matrix $T:M^3\to \mathrm{SO}(2n+4)$ with respect to a fixed basis $\{{\bf e}_k\}_{k=1}^{2n+4}$ of
$\mathbb{R}^{2n+4}$, we can rewrite \eqref{Theta1} as
\begin{equation}\label{Theta2}
dT=\Theta T.
\end{equation}
The algebraic form of $\Theta$ motivates us to introduce a complex structure ${\bf J}$ on $\mathbb{R}^{2n+4}= \mathrm{Span}_\mathbb{R}\{\eta,
Y_3,Y_1,Y_2,\xi_1,\xi_2, \cdots, \xi_{2n-1}, \xi_{2n}\}$ as below:
\[
{\bf J}\begin{pmatrix}\eta\\Y_3 \\ Y_1\\ Y_2\\ \xi_1\\ \xi_2\\ \vdots \\ \xi_{2n-1} \\ \xi_{2n}\end{pmatrix}
=\begin{pmatrix}
\begin{pmatrix}0& -1\\1 & 0\end{pmatrix} & & & &\\
& \begin{pmatrix}0& -1\\1 & 0\end{pmatrix} & & &\\
& &\begin{pmatrix}0& -1\\1 & 0\end{pmatrix}& &\\
& & & \ddots &\\
& & & &\begin{pmatrix}0& -1\\1 & 0\end{pmatrix}
\end{pmatrix}\begin{pmatrix}\eta\\Y_3 \\ Y_1\\ Y_2\\ \xi_1\\ \xi_2\\ \vdots \\ \xi_{2n-1} \\ \xi_{2n}\end{pmatrix}.
\]
Denote the diagonal matrix at the right hand side as $J_0$. Then the matrix representation of operator ${\bf J}$ under
$\{{\bf e}_k\}_{k=1}^{2n+4}$ is:
\[
J=T^{-1}J_0T.
\]
Using $dT=\Theta T$ and the fact that $J_0$ commutes with $\Theta$, it is easy to verify
\[
dJ=-T^{-1}dT T^{-1}J_0T+T^{-1}J_0 dT
=-T^{-1}\Theta J_0T+T^{-1}J_0 \Theta T=0.
\]
So ${\bf J}$ is a well-defined complex structure on
this $\mathbb{R}^{2n+4}$.

Another way to look at the structure equations \eqref{Theta1}
is to consider the complex version. We define
$$\mathcal{Z}_1=\eta+iY_3,\mathcal{Z}_2=Y_1+Y_2,\zeta_1=\xi_1-i\xi_2,\cdots, \zeta_n=\xi_{2n-1}-i\xi_{2n}.$$
Then the complex version of the equation \eqref{Theta1} is
\begin{equation}
\begin{split}
d\mathcal{Z}_1&=-iL\omega_3\mathcal{Z}_1+L(\omega_1-i\omega_2)\mathcal{Z}_2,\\
d\mathcal{Z}_2&=-L(\omega_1+i\omega_2)\mathcal{Z}_1-i\omega_{12}\mathcal{Z}_2+i\mu(\omega_1-i\omega_2)\zeta_1,\\
d\zeta_1&=i\mu(\omega_1+i\omega_2)\mathcal{Z}_2+i\theta_{12}\zeta_1+ia_0(\omega_1-i\omega_2)\zeta_2,\\
d\zeta_k&=ia_{k-2}(\omega_1+i\omega_2)\zeta_{k-1}+i\theta_{{2k-1}\,{2k}}\zeta_k\\
&\quad +ia_{k-1}(\omega_1-i\omega_2)\zeta_{k+1},~~2\leq k \leq n-1, \\
d\zeta_n&=ia_{n-2}(\omega_1+i\omega_2)\zeta_{n-1}+i\theta_{{2n-1}\,{2n}}\zeta_n.
\end{split}
\end{equation}
Geometrically, this implies that
\[
\mathbb{C}^{n+2}=\mathrm{Span}_{\mathbb{C}}\{\mathcal{Z}_1,
\mathcal{Z}_2,\zeta_1, \zeta_2, \cdots,\zeta_n\},
\]
is a fixed n+2 dimensional complex vector space endowed with the complex structure $i$, which is identified with $(\mathbb{R}^{2n+4},{\bf J})$ via the following isomorphism between
complex linear spaces:
\[
v\in \mathbb{C}^{n+2}~~\mapsto~~\mathrm{Re}(v)\in \mathbb{R}^{2n+4}.
\]
For example, $\eta+iY_3\mapsto \eta,i\eta-Y_3\mapsto -Y_3$ and
so on.

The second geometrical conclusion is an interpretation of \eqref{5.1} that $[\eta+iY_3]$
defines a holomorphic mapping from the quotient surface $\overline{M}^2=M^3/\Gamma$ to the projective space $\mathbb{C}P^{n+1}$.
Moreover, the unit circle in
\[
\mathrm{Span}_{\mathbb{R}}\{\eta,Y_3\}=\mathrm{Span}_{\mathbb{C}}\{\eta+iY_3\}
\]
is a fiber of the Hopf fibration of $\mathbb{S}^{2n+3}\subset (\mathbb{R}^{2n+4},{\bf J})$. In fact it corresponds to the subspace
$\mathrm{Span}_{\mathbb{R}}\{Y,\hat{Y},Y_3\}$, which is geometrically
a leave of the foliation $(M^3,\Gamma)$. To see this, it follows from the equations of $d(\xi_{2k-1}-i\xi_{2k})$ that $\{\xi_1, \xi_2, \cdots, \xi_{2n}, d\xi_1, d\xi_2, \cdots , d\xi_{2n}\}$ span a (2n+2)-dimensional spacelike subspace, the corresponding 2-parameter family of 3-dimensional mean curvature sphere congruence has an
envelop $M^3$, whose points correspond to the light-like directions in the orthogonal complement $\mathrm{Span}\{Y,{\hat Y}, Y_3\}$. In particular $[Y],[{\hat Y}]$ are two points on this circle and such circles form a 2-parameter family, with $\overline{M}$ as the parameter space, they give a foliation of $M^3$ which is also a circle fibration.
Thus the whole $M^3$ is the Hopf lift of $\overline{M}^2\to\mathbb{C}P^{n+1}$.
In other words we have the following commutative diagram
\begin{equation*}
\begin{xy}
(30,30)*+{M^3}="v1", (60,30)*+{\mathbb{S}^{2n+3}}="v2", (90,30)*+{\mathbb{C}^{n+2}}="v3";%
(30,0)*+{\overline{M}^2}="v4", (60,0)*+{\mathbb{C}P^{n+1}}="v5".%
{\ar@{->}^{\eta} "v1"; "v2"}%
{\ar@{->}^{\subset} "v2"; "v3"}%
{\ar@{->}_{M^3/\Gamma} "v1"; "v4"}%
{\ar@{->}_{[\eta+iY_3]} "v1"; "v5"}%
{\ar@{->}^{} "v4"; "v5"}%
{\ar@{->}^{\pi} "v2"; "v5"}%
{\ar@{->}_{\pi} "v3"; "v5"}%
\end{xy}.
\end{equation*}
Next we prove $\overline{M}^2\to\mathbb{C}P^{n+1}$ is the Veronese surface of $\mathbb{C}P^{n+1}$. Note that any fibre of the Hopf fibration has the $\mathbb{S}^1$ homogenous. So $\overline{M}^2\to\mathbb{C}P^{n+1}$ must be homogenous and hence has constant curvature. The conclusion follows from the classical results of Calabi(\cite{Calabi} \cite{BJ}).
\end{proof}
Combining Proposition~\ref{the2}, Proposition~\ref{integ2} and Proposition~\ref{integ3}, we finish the proof of our main Theorem\ref{the1}.

\vspace{5mm} \noindent Tongzhu Li,
{\small\it Department of Mathematics, Beijing Institute of
Technology, Beijing 100081, People's Republic of China.
e-mail:{\sf litz@bit.edu.cn}}

\vspace{5mm} \noindent Xiang Ma,
{\small\it LMAM, School of Mathematical Sciences, Peking University,
Beijing 100871, People's Republic of China.
e-mail: {\sf maxiang@math.pku.edu.cn}}

\vspace{5mm} \noindent Changping Wang,
{\small\it School of Mathematics and Computer Science,
Fujian Normal University, Fuzhou 350108, People's Republic of China.
e-mail: {\sf cpwang@fjnu.edu.cn}}

\vspace{5mm} \noindent Zhenxiao Xie,
{\small\it School of Mathematical Sciences, Peking University,
Beijing 100871, People's Republic of China.
e-mail: {\sf xiezhenxiao@126.com}}

\end{document}